\documentclass[11pt]{amsart}
\usepackage{amsmath,amsfonts,amssymb,amsthm}
\usepackage{graphics}
\usepackage{graphicx}
\usepackage{overpic}
\usepackage{epsfig}
\usepackage{color}
\usepackage{array}
\usepackage{url}
\usepackage{epic}
\usepackage{hyperref}

\newtheoremstyle{thm}
  {9pt}{9pt}{\itshape}{}{\bfseries}{}{.5em}{}

\makeatletter
\newcommand{\longdash}[1][2em]{%
  \makebox[#1]{$\m@th\smash-\mkern-7mu\cleaders\hbox{$\mkern-2mu\smash-\mkern-2mu$}\hfill\mkern-7mu\smash-$}}
\makeatother
\newcommand{\omitskip}{\kern-\arraycolsep}

\theoremstyle{thm}
\newtheorem{thm}{Theorem}[section]
\newtheorem{cor}[thm]{Corollary}
\newtheorem{lemma}[thm]{Lemma}

\theoremstyle{remark}

\theoremstyle{thm}
\newtheorem{proposition}[thm]{Proposition}
\newtheorem{theorem}[thm]{Theorem}
\newtheorem{corollary}[thm]{Corollary}
\theoremstyle{remark}
\newtheorem{remark}[thm]{Remark}

\newtheoremstyle{defin}
  {9pt}{9pt}{}{}{\bfseries}{}{.5em}{}
\theoremstyle{defin}
\newtheorem*{defin}{Definition}
\newtheorem*{definition}{Definition}

\newtheorem{cl}{Claim}
\newtheorem{construction}{Construction}
\newtheoremstyle{exm}
  {9pt}{9pt}{}{}{\scshape}{}{.5em}{}
\theoremstyle{exm}

\newtheoremstyle{proof}
  {}{}{}{}{\itshape}{:}{.5em}{}
\theoremstyle{proof}

\newcommand{\Z}{{\mathbb Z}}
\newcommand{\R}{{\mathbb R}}

\DeclareMathOperator{\conv}{Conv}

\def\SS{\mathbb S}
\def\Sph{\mathbb S}

\def\sm{\smallsetminus}

\def\T{\mathbf{T}}

\def\BB{\mathbf{B}}

\def\<{\langle}
\def\>{\rangle}

\def\RR{ {\text {\rm R} } }

\def\0{{\mathbf 0}}

\def\TT{{\mathcal T}}

\def\.{\hskip.06cm}

\def\conv{{\text {\rm {conv}} }}

\newcommand{\V}{\mathcal{V}}

\renewcommand{\S}{\mathcal{S}}

\newcommand{\N}{\mathbb{N}}

\title{Many triangulated odd-spheres}
\author{Eran Nevo, Francisco Santos and Stedman Wilson}
\address
[E.~Nevo]
{
Department of Mathematics,
Ben Gurion University of the Negev,
Be'er Sheva 84105, Israel
}
\email{nevoe@math.bgu.ac.il}
\address
[F.~Santos]
{
Departamento de Matem\'aticas, Estad\'istica y Computaci\'on,
Universidad de Cantabria,
39005 Santander, Spain
}
\email{francisco.santos@unican.es}
\address
[S.~Wilson]
{
Department of Mathematics,
Ben Gurion University of the Negev,
Be'er Sheva 84105, Israel
}
\email{stedmanw@gmail.com}

\thanks{Research of Nevo and Wilson was partially supported by Marie Curie grant IRG-270923 and ISF grant 805/11.
Research of Santos was supported by the Spanish Ministry of Science (MICINN) through grant MTM2011-22792, and by a Humboldt Research Award of the Alexander von Humboldt Foundation.
}

\begin{document}

\begin{abstract}
It is known that the $(2k-1)$-sphere has at most $2^{O(n^k \log n)}$ combinatorially distinct triangulations with $n$ vertices, for every $k\ge 2$. Here we construct at least $2^{\Omega(n^k)}$ such triangulations, improving on the previous constructions which gave $2^{\Omega(n^{k-1})}$ in the general case (Kalai) and $2^{\Omega(n^{5/4})}$ for $k=2$ (Pfeifle-Ziegler).

We also construct $2^{\Omega\left(n^{k-1+\frac{1}{k}}\right)}$ geodesic (a.k.a. star-convex) $n$-vertex triangualtions of the $(2k-1)$-sphere.
As a step for this (in the case $k=2$) we  construct $n$-vertex $4$-polytopes containing $\Omega(n^{3/2})$ facets that are not simplices,
or with $\Omega(n^{3/2})$ edges of degree three.
\end{abstract}

\maketitle
\tableofcontents

\section{Introduction}
For $d\geq 3$ fixed and $n$ large,
Kalai \cite{Kalai-manyspheres} constructed $2^{\Omega(n^{\lfloor d/2 \rfloor})}$ combinatorially distinct
 $n$-vertex triangulations of the $d$-sphere (the \emph{squeezed} spheres), and concluded from
Stanley's upper bound theorem for simplicial spheres~\cite{Stanley:CohenMacaulayUBC-75}
an upper bound of $2^{O(n^{\lceil d/2 \rceil}\log n)}$ for the number of such triangulations.  In fact, this upper bound readily follows from the Dehn-Sommerville relations, as they imply that the number of $d$-dimensional faces is a linear combination of the number of faces of dimension $\leq \lceil d/2\rceil-1$, and hence is at most $O(n^{\lceil d/2\rceil})$. Thus, as already argued in \cite{Kalai-manyspheres}, the number of different triangulations is at most
$\binom{\binom{n}{d+1}}{O(n^{\lceil{d/2}\rceil})}$, namely at most $2^{O(n^{\lceil d/2 \rceil}\log n)}$.

For even $d$ the difference between the upper and lower bound is only a $\log n$ in the exponent, but in odd dimension $d=2k-1$ the gap is much bigger, from $2^{\Omega(n^{k-1})}$ to $2^{O(n^{k}\log n)}$. Most strikingly, for $d=3$ the gap is from $2^{\Omega(n)}$ to $2^{O(n^{2}\log n)}$.
Pfeifle and Ziegler \cite{Pfeifle-Ziegler:many3-spheres} reduced this gap by constructing $2^{\Omega(n^{5/4})}$ combinatorially different $n$-vertex triangulations of the $3$-sphere. Our main result is a construction that gives $2^{\Omega(n^k)}$ combinatorially different $(2k-1)$-spheres, close to the upper bound of $2^{\Omega(n^k\log n)}$.

The bound in \cite{Pfeifle-Ziegler:many3-spheres} is obtained by constructing a polyhedral $3$-sphere with $\Omega(n^{5/4})$ combinatorial octahedra among its facets.
Our bound will follow from constructing a polyhedral $3$-sphere with $\Omega(n^2)$ combinatorial \emph{bipyramids} (or their natural generalization to higher dimension) among its facets. A bipyramid is the unique simplicial $3$-polytope with $5$ vertices.
The idea of the construction, detailed in Section~\ref{sec:general} in a more general setting (see, in particular, Theorem~\ref{thm:general}), is as follows: consider a certain simplicial  $3$-ball $K$ with $n$ vertices.
\begin{itemize}
\item{}Find particular $\Theta(n)$ simplicial $3$-balls contained in $K$, with disjoint interiors.
\item{}On the boundary of each such $3$-ball find particular $\Theta(n)$ pairs of adjacent triangles  (each pair forms a square), such that these squares have disjoint interiors.
\item{}Replace the interior of each such $3$-ball with the cone from a new vertex over each boundary square (forming a bipyramid) and over each remaining boundary triangle (forming a tetrahedron).
\item{}Show that the particular $3$-balls and squares chosen have the property that the above construction results in a polyhedral $3$-ball. Adding a cone over the boundary results in a desired polyhedral $3$-sphere.
\end{itemize}
%%%%%%%%%%%%%%%%%%%%%%%%%%%

Specifically, we prove the following:

\begin{thm}[Theorems~\ref{thm:holes4} and~\ref{thm:holes4_cyclic}]
There is a constant $c>0$ such that for every $n\geq 4$ there exists a $3$-dimensional polyhedral sphere $S_n$ with $n$ vertices and at least $cn^2$ facets that are combinatorially equivalent to bipyramids.
\label{t:complex}
\end{thm}

Note that each of the bipyramids in the above theorem can be triangulated  in two ways --- either into $2$ tetrahedra by inserting its missing triangle or into $3$ tetrahedra by inserting its missing edge --- to obtain a triangulation of the $3$-sphere.
The fact that all missing faces are different implies that we can triangulate the bipyramids independently and always get a valid simplicial complex.
Moreover, all the steps in the process are done in the PL-category, also in higher dimensions.
(All simplicial $3$-spheres are PL by Moise's theorem \cite{Moise}, however non-PL simplicial spheres exist in any dimension $\geq 5$.)
This gives the following result.

\begin{cor}[Corollaries~\ref{cor:holes4_tri} and~\ref{cor:holes4_cyclic}]
The ({\rm PL}-) sphere of dimension $3$ admits $2^{\Omega(n^2)}$ combinatorially distinct triangulations on $n$ vertices.
\label{t:sphere}
\end{cor}

We show two specific constructions providing Theorem~\ref{t:complex}, one based in the join of two paths (Theorem~\ref{thm:holes4})
and one based in the boundary complex of the cyclic $4$-polytope (Theorem~\ref{thm:holes4_cyclic}). The latter gives a better constant inside the $\Theta(\cdot)$ notation ($4n^2/25$, versus $2n^2/25$ bipyramids in the former) but the former is somehow simpler to describe and serves as a preparation for the latter.

Erickson conjectured that there are no $4$-polytopes or $3$-spheres on $n$ vertices with $\Omega(n^2)$ non-simplicial facets.
Theorem \ref{t:complex} refutes this for $3$-spheres, but we leave the question open for $4$-polytopes. For them we can only prove the following, which is the first construction of $4$-polytopes with more than $O(n^{1+\epsilon})$ non-simplicial facets (a construction of cubical $4$-polytopes with $\Theta(n \log(n))$ non-simplicial facets is due to Joswig and Ziegler \cite{Joswig-Ziegler}):

\begin{theorem}[Theorem~\ref{thm:holes_aztec}, Corollary~\ref{coro:holes_aztec}]
\label{t:polytope}
There are $4$-dimensional polytopes with $n$ vertices and with $\Theta(n^{3/2})$ facets that are bipyramids.
\end{theorem}

Triangulating these bipyramids appropriately we also get:

\begin{theorem}[Corollary~\ref{coro:degree_3}]
There are $4$-dimensional simplicial polytopes with $n$ vertices and with $\Theta(n^{3/2})$ edges of degree three.
\end{theorem}

This answers the following question of Ziegler (see, for example,~\cite{Kalai-Kyoto}): Can a simple 4-polytope with $n$ facets have more than $O(n)$ non-quadrilateral 2-faces? Indeed, the dual to the polytopes in the above theorem have $n$ facets and $\Theta(n^{3/2})$ triangles.

The spheres of Theorem~\ref{t:polytope}  can be triangulated in $2^{\Theta(n^{3/2})}$ ways. These
triangulations cannot be all polytopal, simply because they are too many. But they are \emph{geodesic}, which is an intermediate class between polytopal spheres and the class of all triangulated spheres.  By a geodesic sphere (sometimes also called a \emph{star-convex} sphere) we mean one that can be realized geodesically in the standard sphere $\Sph^{d}\subset\R^{d+1}$. Put differently, a simplicial complex is a geodesic $d$-sphere if it can be realized as a complete simplicial fan in $\R^{d+1}$.

\begin{theorem}[Theorem~\ref{thm:holes_aztec}, Corollary~\ref{coro:holes_aztec}]
There are at least $2^{\Theta(n^{3/2})}$ geodesic $3$-spheres with $n$ vertices.
\end{theorem}

In Section~\ref{sec:higherdim} we extend the above theorems  to higher odd-dimensional spheres. To state our general results and to put them in the context, let us define:
\begin{eqnarray*}
s_d(n)&:=&\text{number of n-vertex simplicial spheres of dimension $d$}, \\
p_d(n)&:=&\text{number of n-vertex  \emph{polytopal} simplicial spheres of dimension $d$},\\
g_d(n)&:=&\text{number of n-vertex  \emph{geodesic} simplicial spheres of dimension $d$}.
\end{eqnarray*}

Obviously,
\[
p_d(n) \le g_d(n) \le s_d(n).
\]

Previous results show the following asymptotics of these functions. The result for $p_d(n)$ is from Goodman and Pollack~\cite{GoodmanPollack:fewPolytopes-86, Alon:fewPolytopes} and the bounds for $s_d(n)$ are, as mentioned at the beginning, from Kalai~\cite{Kalai-manyspheres} and the Upper Bound Theorem.
\[\log p_d(n) \in \Theta(n\log n).\]
 \[
 \Omega(n^{\lfloor d/2 \rfloor})\le \log s_d(n)\le O(n^{\lceil d/2 \rceil}\log n).
 \]
Concerning $g_d(n)$, we do not know of any explicit previous bounds, but the following two are implicit in Theorems 6.1.22 and 8.4.2(3) of~\cite{deLoeraRambauSantos}. These theorems, in turn, follow from ideas of Dey~\cite{Dey93:triangulations} and Kalai~\cite{Kalai-manyspheres}. (In fact, the proof of~\cite[Theorem 6.1.22]{deLoeraRambauSantos} can easily be adapted to show that Kalai's spheres are  geodesic).

\begin{lemma}
\label{lemma:dey}
For fixed $d$,
\[
 \Omega(n^{\lfloor d/2 \rfloor})\le \log g_d(n)\le O(n^{\lfloor d/2 \rfloor+1}).
\]
\end{lemma}

\begin{proof}
For the upper bound, we look at geodesic $d$-spheres as complete simplicial fans in $\R^{d+1}$. To construct one such fan with $n$ rays we first fix a set $V$ of $n$ vectors in $\R^{d+1}$ and then look at the number of simplicial fans having rays exactly in the directions of $V$. (These are the \emph{triangulations of $V$} in the terminology of~\cite{deLoeraRambauSantos}).

It is well-known that two configurations $V$ and $V'$ having the same oriented matroid produce isomorphic triangulations~\cite[Corollary 4.1.44]{deLoeraRambauSantos}. So, the number of ``different'' configurations we need to consider is bounded by the number of (uniform, totally unimodular) realizable oriented matroids of rank $d+1$ on $n$ elements. This is bounded above by $2^{O(n\log n)}$~\cite{GoodmanPollack:fewPolytopes-86, Alon:fewPolytopes}.

Once $V$ is fixed, a triangulation of $V$ is determined by its $\lfloor d/2 \rfloor$-skeleton (\cite[Lemma 8.4.1]{deLoeraRambauSantos}), so the number of them is at most $2^{n \choose \lfloor d/2 \rfloor +1} \in 2^{O(n^{\lfloor d/2 \rfloor +1})}$. Multiplying this by $2^{O(n\log n)}$ does not affect it.

For the lower bound, we take as starting point the fact that the cyclic $d$-polytope with $n$ vertices has $2^{ \Omega(n^{\lfloor d/2 \rfloor})}$ triangulations~\cite[Theorem 6.1.22]{deLoeraRambauSantos}. Central projection of $\R^d$ to an open hemisphere in $\SS^d$, followed by coning the boundary of each triangulation to a point in the antipodes, produces that many geodesic spheres.
\end{proof}

In even dimensions the difference between the lower and upper bounds for $s_d$ is not that big, and the upper bound for $g_d$ is not an improvement of the upper bound for $s_d(n)$. But in odd dimension $2k-1$ the bounds become:
\[
 \Omega(n^{k-1})\le \log s_{2k-1}(n)\le O(n^{ k }\log n).
 \]
 \[
 \Omega(n^{k-1}) \le  \log g_{2k-1}(n)\le O(n^{ k}).
\]
We improve the two lower bounds as follows:

\begin{theorem}[Corollaries~\ref{coro:high_d} and~\ref{t:sphere_k}]
\[
\log s_{2k-1}(n) \ge \Omega(n^{ k}),
\qquad
\log g_{2k-1}(n) \ge  \Omega(n^{ k-1+ \frac{1}{k}}).
\]
\end{theorem}

These bounds follow from the construction of a $2k-1$-sphere with $n$ vertices and at least $\frac{2n^k}{3k^{k+1}}$ non-simplicial facets (Theorem~\ref{thm:high_d}) and a geodesic one with at least $\frac{2n^{k-1+\frac{1}{k}}}{(k-1)!(k+1)^{k}}$ non-simplicial facets (Theorem~\ref{thm:aztec-highd}).
Moreover, all the triangulations that refine these spheres are PL-spheres. We leave as an open question whether they are shellable.

Another interesting open question is whether $\lim _{n\to \infty}\log g_d(n) / \log s_d(n)$ is, for fixed $d$, zero or positive. Observe that individual non-geodesic spheres are easy to construct; for example, a $3$-sphere containing a trefoil knot on five edges or less cannot be geodesic.

%%%%%%%%%%%%%%%%%%%%%%%%%%%%%%%%%%%%%%%%%%%%%%%
\section{Preliminaries}

For background on polytopes, the reader can consult~\cite{Ziegler} or~\cite{Grunbaum}, and for background on simplicial complexes
see e.g.~\cite{Munkres}.

Let $X$ denote a complex, simplicial or polyhedral.
In what follows, always assume $X$ to be \emph{pure}, namely all its maximal faces with respect to inclusion have the same dimension.  We write $\V(X)$ for the set of all vertices of $X$.  By a \emph{facet} of $X$ we mean a face $F \in X$ maximal w.r.t. inclusion.  We use the term $k$-face as a shorthand for $k$-dimensional face, as usual.  Similarly, we call a complex $X$ a \emph{$k$-complex} if all of its facets are $k$-faces.

For a polyhedral ball $X$ we denote by $\partial X$ the \emph{boundary complex} of $X$.  That is, $\partial X$ is the subcomplex of $X$ whose facets are precisely the faces of $X$ that are contained in exactly one facet of $X$.  In particular, if $X$ is a $k$-complex homeomorphic to the  $k$-ball, then $\partial X$ is $(k - 1)$-dimensional, homeomorphic to the $(k-1)$-sphere.  We say that a face $F \in X$ is \emph{interior} to a polyhedral ball $X$ if $F \notin \partial X$.
For $X$ a simplicial complex, the \emph{link} of a face $F\in X$ is the subcomplex
$\{T\in X:\ F\cap T=\emptyset, \ F\cup T\in X\}$, and its (closed) \emph{star} is the subcomplex generated by the faces
$\{T\in X:\ F\subseteq T\}$ under taking subsets.
If $P$ is a \emph{simplex}, then we will identify $P$ with its set of vertices $\V(P)$, or with the simplicial complex $2^{\V(P)}$, when convenient.

The \emph{join} of two simplicial complexes $X$ and $Y$ is defined as
\[
X*Y:=\{F\cup G :\ F\in X, G\in Y\}.
\]
Observe that the star of $F$ equals the join of $F$ and the link of $F$.

We will make use of the following arithmetic notation.  For any integer $n \geq 1$, we will use $[n]$ as a shorthand for $[1, n] \cap \Z$, the set of all integers from $1$ to $n$.  For a real number~$r$, we will denote by $\lfloor r \rfloor$ the floor of $r$, and by $\lceil r \rceil$ the ceiling of $r$.

One of the foundations for our construction is the \emph{cyclic polytope}, so we restate its definition here.  The \emph{moment curve} in $\R^d$ is the curve $\alpha_d : \R \rightarrow \R^d$ defined by
\[
\alpha_d(t) =  (t, t^2, t^3, \dots, t^d).
\]

The convex hull of the image of $[n]$ under $\alpha_d$, which we will denote by $C(n, d)$, is the \emph{cyclic $d$-polytope} with the $n$ vertices $\alpha_d(1), \alpha_d(2), \ldots, \alpha_d(n)$.  The faces of the cyclic polytope admit a simple combinatorial description, called \emph{Gale's evenness condition} (see e.g.~\cite{Grunbaum}, p.~62, and~\cite{Ziegler}, p.~14).  We restate this property here as a lemma.

\begin{lemma} All facets of $C(n, d)$ are $(d - 1)$-simplices.  Furthermore, for any set of $d$ integers $I \subset [n]$, the convex hull $\conv(\alpha_d(I))$ is a facet of $C(n, d)$ if and only if for every $x, y \in [n] \sm I$, there are an even number of elements $z \in I$ satisfying $x < z < y$.
\label{l:cyclic}
\end{lemma}

We only consider cyclic polytopes of even dimension $d=2k$. Their boundary complex, as an abstract simplicial complex on $[n]$, is symmetric under the natural action of the dihedral group on $[n]$. But another important property for our purposes is that if $P_1,\dots,P_k$ are a decomposition of the path $[1,n]$ into subpaths then
\[
P_1 * \dots * P_k \subset \partial C(n, 2k).
\]

Let $\Delta^d$ denote the $d$-simplex. Place $\Delta^{k - 1}\subseteq \R^{k-1}\times \{0\}\subseteq \R^{2k-1}$ and
 $\Delta^{k}\subseteq \{0\}\times \R^{k} \subseteq \R^{2k-1}$ such that the origin is in the relative interior of both simplices.
The \emph{free sum} $\Delta^{k - 1} \oplus \Delta^{k}$ is the simplicial polytope combinatorially equivalent to $\conv(\Delta^{k - 1}\cup \Delta^{k})$ under the above placing. Thus a bipyramid is the free sum $\Delta^{1}\oplus \Delta^{2}$. Observe that
\[
\partial(\Delta^{l}\oplus \Delta^{k})= \partial \Delta^{l} * \partial \Delta^{k}.
\]
Observe also that any simplicial sphere (or homology manifold) of dimension $d$ on $d+3$ vertices is a join of two boundaries of simplices. One can show this by induction on $d$, noticing that any vertex link is either the boundary of a $d$-simplex, or a (homology) $(d-1)$-sphere with $d+2$ vertices.

\section{Carving and filling holes in simplicial complexes}
\label{sec:general}

Let $X$ be a pure simplicial complex of a certain dimension $d$. Let $B_1,\dots, B_k$ be subcomplexes of $X$, each of which is a simplicial $d$-ball and which have disjoint interiors. This is equivalent to no $d$-simplex belonging to two different $B_i$'s.
Under these conditions:

\begin{lemma}\label{lem:3.1}
Adding $k$ new vertices $v_1,\dots , v_k$ to $X$ and substituting each $B_i$ for $\partial B_i * v_i$ produces a simplicial complex $X'$ homeomorphic to $X$.
\end{lemma}

We want to modify the above construction so that the new complex $X'$, instead of being simplicial, contains many non-simplicial cells. Since the most economic non-simplicial cells are free sums of simplices, these are the ones that we are interested in. We concentrate first in a single ball $B$, then treat the case of several balls $B_1,\dots,B_k$.
Throughout, assume $B$, and each $B_i$, are not a simplex.

For each $d$-simplex $\sigma$ in a simplicial $d$-ball $B$, denote $D_\sigma$ the set of all facets of $\sigma$ that are not interior to $B$.
Put differently, $D_\sigma$ is the pure $(d-1)$-dimensional part of $\partial \sigma \cap \partial B_i$.
Some $D_\sigma$'s may be empty, but for those which are not empty we have the following properties:
\begin{itemize}
\item For $\sigma \ne \tau$ we have that $D_{\sigma} \cap D_{\tau}=\emptyset$,
as each $(d-1)$-simplex in the boundary of a $d$-ball $B$ is a facet of a unique $d$-simplex in $B$.
\item Each $D_\sigma$ is a $(d-1)$-ball, since every pure $(d-1)$-subcomplex of the boundary of a $d$-simplex is a $(d-1)$-ball. (An exception is if $D_\sigma= \partial \sigma$, but this can only happen if $B$ is a single simplex $\sigma$, which we assume not to be the case).
\item Each $D_\sigma$ is the complement in $\partial \sigma$ of the (open) star of a face $F_\sigma$ of $\sigma$. We call $F_\sigma$ the \emph{missing face} of $D_\sigma$, being the only subset of $V(D_\sigma)$ not in $D_\sigma$ so that all its proper subsets are in $D_\sigma$.
\end{itemize}

We introduce the following definition:

\begin{defin}
A family $\{\sigma_1,\dots,\sigma_t\}$ of $d$-simplices in a $d$-ball $B$ is \emph{compatible} if the missing faces $F_{\sigma_1},\dots,F_{\sigma_t}$ are all different and interior to $B$ (that is, no $F_i$ is contained in $\partial B$).
\end{defin}

\begin{lemma}
\label{lemma:one_hole}
Let $S$ be a compatible family of $d$-simplices in a simplicial $d$-ball $B$. Let $v$ be a new vertex.
Consider the following polyhedral complex:
\begin{enumerate}
\item For each $\sigma\in S$, consider the cell $C_\sigma$ with boundary complex $D_\sigma \cup (\partial D_\sigma * v)$. This is (combinatorially) a free sum of two simplices of dimensions $|F_\sigma|-1$ and $d+1 - |F_\sigma|$, where $|F_\sigma|$ denotes the number of vertices of the missing face $F_\sigma$ of $D_\sigma$ (equivalently, $|F_\sigma|$ is the number of facets of $\sigma$ that make $D_\sigma$).

\item For each $(d-1)$-simplex $\tau$ of $\partial B$ that is not in any $D_\sigma$, $\sigma\in S$, consider the simplex $\tau*v$.
\end{enumerate}
The complex $B'$ so obtained is a polyhedral complex homeomorphic to $B$.
\end{lemma}

\begin{proof}
Topologically, $B'$ is obtained from $\partial B * v$ by merging the simplices that come from the same $D_\sigma$. Hence, it is indeed a ball. Let us check that $C_\sigma$ is indeed as described, and that  different cells intersect properly.

For the description of $\sigma$, let $G_\sigma$ be the face of $\sigma$ complementary to $F_\sigma$. Equivalently, $G_\sigma$ is the intersection of all facets of $\sigma$ in the boundary of $B$, and the unique minimal face of $D_\sigma$ that is not in $\partial D_\sigma$. Then, $C_\sigma$ is the free sum of the simplices $F_\sigma$ and $G_\sigma*v$, because these two are the unique minimal non-faces in the boundary complex of $C_\sigma$.

For the proper intersection there are three cases, depending on whether each of the two cells is of the form $\tau*v$ or $C_\sigma$.
\begin{itemize}
\item The case $(\tau*v)\cap (\tau'*v)$ is trivial.
\item The case $(\tau*v)\cap C_\sigma$ is easy. Consider the intersection $F=\tau\cap \sigma$, which is a face of both $\tau$ and $\sigma$ of dimension at most $d-2$. Since $\tau\subset \partial B$, $F$ is contained in the boundary of $D_\sigma$, and as $\tau\notin D_\sigma$, we get $F\in \partial D_\sigma$.
    Than,
$(\tau*v)\cap C_\sigma = F*v$, and is a face of both.

\item The case $C_\sigma\cap C_{\sigma'}$ is similar. Consider the intersection $F=\sigma\cap \sigma'$, which is a face of both $\sigma$ and $\sigma'$.

For $C_\sigma$ and $C_{\sigma'}$ not to intersect properly, it is necessary that $F_\sigma \cup F_{\sigma'}\subset F$. Indeed, if $F_\sigma\not\subset F$ then $F$ is a face of $D_\sigma$. This implies $F\subset \partial B$, so in particular $F_{\sigma'}\not\subset F$ either, and $F$ is also a face of $D_{\sigma'}$. This face $F$ must be in the boundary of both, so that
$\partial D_\sigma * v$ and $\partial D_{\sigma'} * v$
intersect in $F*v$, and $F*v$ is a face in both $C_\sigma$ and $C_{\sigma'}$.

So, let us assume by way of contradiction that $F_\sigma \cup F_{\sigma'}\subset F = \sigma\cap \sigma'$. In particular, $F_\sigma$ is a face of $\sigma'$ and $F_{\sigma'}$ a face of $\sigma$. Now, every face of $\sigma'$ not contained in $\partial B$ must contain $F_{\sigma'}$, so $F_{\sigma'}\subset F_\sigma$. For the same reason, $F_\sigma \subset F_{\sigma'}$. Thus,
$F_\sigma = F_{\sigma'}$, a contradiction to compatibility.
\end{itemize}
\end{proof}

Since this construction is the basis of all our results we give a name to it:

\begin{defin}
Let $S$ be a compatible family of $d$-simplices in a simplicial $d$-ball $B$. The process leading to the ball $B'$ of Lemma~\ref{lemma:one_hole}
is called \emph{$S$-filling $B$}.
\end{defin}

We now come to the main result of this section, in which we apply the previous lemma to several balls simultaneously:

\begin{thm}
\label{thm:general}
Let $K$ be a pure simplicial $d$-complex with $n$ vertices. Let $B_1,\dots, B_k$ be $d$-balls contained in $K$ and with disjoint interiors. Let $S_1,\dots,S_k$ be compatible families of $d$-simplices, each $S_i$ contained in the corresponding $B_i$. Consider the complex $K'$ obtained from $K$ by $S_i$-filling $B_i$ for every $1\leq i \leq k$. Let $b=\sum_i |S_i|$. Then:
\begin{enumerate}
\item $K'$ is a polyhedral complex homeomorphic to $K$ with $n+k$ vertices whose $d$-cells are a certain number of simplices plus $b$ free sums of two simplices (each of dimension at least $1$).
\item Each of the $b$ free sums of two simplices admits two triangulations without new vertices. The $2^b$ possibilities to independently triangulate them all produce simplicial complexes with $n+k$ vertices and homeomorphic to $K$.
\item If $B_1,\dots,B_k$ are {\rm PL}-balls then the $2^b$ triangulations of the previous statement are {\rm PL}-homeomorphic to $K$.
\end{enumerate}
\end{thm}

\begin{proof}
For part (1), the homeomorphism follows from Lemma \ref{lem:3.1}.
 The only thing missing, after we have Lemma~\ref{lemma:one_hole}, is that cells from different $B'_i$'s intersect properly.
The only non-obvious case is that of cells coming from a $C_\sigma$ and a $C_{\sigma'}$, with $\sigma\in B_i$ and $\sigma'\in B_j$.
Now, in that case the intersection is obviously $D_\sigma\cap D_{\sigma'}$, contained in $\sigma\cap \sigma'$,
$(\sigma\cap\partial B_i) \cap (\sigma'\cap \partial B_j)= D_\sigma \cap D_{\sigma'}$ in $K$, thus also in $K'$ (as $\cup_l \partial B_l$ is a subcomplex of $K'$).

For part (2), consider a particular $C_\sigma$ and let, as usual, $F_\sigma$ be the missing face of $D_\sigma$, so that $D_\sigma$ is the complement of the open star of $F_\sigma$, and $G_\sigma$ be the face of $\sigma$ complementary to $F_\sigma$. Since $C_\sigma$ is the free sum of the simplices $F_\sigma$ and $G_\sigma*v_i$, where $v_i$ is the new vertex introduced in the ball $B_i$ containing $\sigma$, these two simplices are, in fact, the two unique minimal non-faces of $\partial C_\sigma$. The two triangulations of $C_\sigma$ are
\[
T_{\sigma,1} = F_\sigma * \partial G_\sigma, \qquad
T_{\sigma,2} = \partial F_\sigma *  G_\sigma.
\]
Now, as $T_{\sigma,i}$ leaves $\partial C_\sigma$ intact, part (1) guarantees that indeed the $2^b$ complexes obtained are triangulations of $K'$.

For part (3), observe first that the $2^b$ triangulations are PL-homeomorphic to one another, since changing the choice of refinement in each individual  free sum of two simplices is a bistellar flip. So, it is enough to prove that one of them is PL-homeomorphic to $K$, and we choose the one obtained by taking a cone over the boundary of each $B_i$. Since $B_i$ is clearly PL-homeomorphic to $\partial B_i * v_i$ (via an homeomorphism that is the identity on $\partial B_i$) changing each $B_i$ to  $\partial B_i * v_i$ for each $i$ does not change  the PL-homeomorphism type of the complex $K$.
\end{proof}

\section{Many $3$-dimensional spheres}
\label{sec:dim3}

In this section we construct $2^{\Theta^{(n^2)}}$ combinatorially different $3$-dimensional spheres on $n$ vertices, and $2^{\Theta^{(n^{3/2})}}$ geodesic ones. For spheres we include two versions of the construction. One based on the join of two paths, which is somehow easier to analyze, and one based on the cyclic polytope, which leads to a better asymptotic constant in the $\Theta(\cdot)$ notation.

\subsection{Many $3$-spheres in the join of two paths}

Let  $a_1,\dots,a_n$ be $n$ points in the line $\{(t,0,1): t\in \R\}$ and let $b_1,\dots,b_m$ be $m$ points in the line $\{(0,t,-1): t\in \R\}$.
Let $A$ be the set of  $n+m$ points:
\[
A=\{a_1,\dots,a_n,b_1,\dots,b_m\}.
\]

The only triangulation of $\conv(A)$ with vertex set $A$ is the \emph{join} of the two paths along the two lines, given by the tetrahedra:
\[
\TT=\{a_ia_{i+1}b_jb_{j+1}: i=1,\dots, n-1, \ j=1,\dots, m-1\}.
\]
Here we identify a tetrahedron with its set of vertices and we sometimes abbreviate $\{a_i,a_{i+1},b_j,b_{j+1}\}$ as $a_ia_{i+1}b_jb_{j+1}$.
One nice way of picturing the three-dimensional complex $\TT$ in 2-dimensions is to cut it through the plane $z=0$. This produces an $(n-1)\times (m-1)$ grid of rectangles, whose $n\times m$ vertices are the midpoints $(a_i+b_j)/2$.

We denote $T_{ij}$ the tetrahedron $a_ia_{i+1}b_jb_{j+1}$. Each subset $B\subset [n-1]\times [m-1]$ represents a pure subcomplex $\TT_B$ of $\TT$:
\[
\TT_B:=\{a_ia_{i+1}b_jb_{j+1}:(i,j)\in B\}.
\]

The following definition and lemma relate combinatorial properties of $B$ to geometric and topological properties of $\TT_B$:

\begin{definition} Let $B\subset[n-1]\times [m-1]$.
\label{defi:grid_convexity}
\begin{enumerate}
\item We say that $B$ is \emph{grid-connected} if we can go from any $(i,j)\in B$ to any $(i',j')\in B$ changing one index at a time and by one unit at a time. That is, if the subgraph induced by $B$ in the cartesian product of two paths is connected.
\item We say that $B$ is \emph{grid-starconvex} from $(i,j)$ if, for every $(i',j')\in B$ and every $i''\in [i,i']$ and $j''\in[j,j']$ we have $(i'',j'')\in B$. Here we are not assuming $i\le i'$ or $j\le j'$, so that $[i,i']$ represents the minimum interval containing $i$ and $i'$, independently of which is smaller.
\item We say that $B$ is \emph{grid-unimodal} if it is grid-connected and $B\cap (\{i\}\times [m-1])$ and $B\cap ([n-1]\times \{j\})$ are connected (that is, intervals), for every $i$ and $j$.
\end{enumerate}
\end{definition}

To state the geometric counterparts of these properties, we denote as $|\TT_B|$ the union of (the convex hulls of) the tetrahedra in $\TT_B$.

\begin{lemma}
\label{lemma:grid_convexity}
Let $B\subset[n-1]\times [m-1]$. Then:
\begin{enumerate}
\item $B$ is {grid-connected} if and only if the interior of $\TT_B$ is connected.
\item $B$ is {grid-starconvex} from $(i,j)$ if and only if $|\TT_B|$ is star convex from every point (equivalently, from some point) in the interior of the tetrahedron $T_{ij}$.
\item $B$ is {grid-unimodal} if and only if $|\TT_B|$ is topologically a $3$-ball. This ball is, moreover, shellable.
\end{enumerate}
\end{lemma}

\begin{proof}
For part (1), observe that the product of two paths is nothing but the dual graph of $\TT$.
Hence, $B$ is grid-connected if and only if the dual graph of $\TT_B$ is connected, which is equivalent to the interior of $|\TT_B|$ being connected.%
\footnote{Here ``interior'' may have two different meanings, but the result and the proof hold in both: on the one hand, $|\TT|$ is a manifolds with boundary, and we can call interior of an $X\subset |\TT|$ the part of the topological interior of $X$ not meeting the boundary of $|\TT|$. This coincides with the topological interior of $X$ as a subset of $\R^3$. Interior may also mean the interior of $X$ in $|\TT|$, which may include part of the boundary of $|\TT|$).
}

For part (2), let $x$ be any point in the interior of $T_{ij}$.
If $B$ is not star convex from $(i,j)$ then there is an $(i',j')\in B$ (without loss of generality assume $i'\ge i$ and $j'\ge j$) such that either $(i'-1,j')$ or $(i',j'-1)$ is not in $B$. Consider a point $y\in |T_{i'j'}|$ very close to the facet of it common to $T_{i'-1,j'}$ or $T_{i',j'-1}$. Clearly, the segment joining $x$ and $y$ is not contained in $|\TT_B|$.
For the converse,
assume $B$ is star convex from $(i,j)$ and
let $y$ be any point in $|\TT_B|$, and assume $y\in |T_{i'j'}|$ for  a certain tetrahedron in $\TT$. Then, it is clear that every other tetrahedron $|T_{i''j''}|$ meeting the segment $xy$ will have $i''\in [i,i']$ and $j''\in[j,j']$, so that $T_{i''j''}\in B$, hence $xy \subseteq |\TT_B|$.

We finally prove part (3). For necessity, if $B$ is not grid-unimodal then either it is not grid-connected (and hence $|\TT_B|$ not a ball, by part (1)) or there is (without loss of generality) a certain $i\in [n-1]$ and certain $j<j''<j'\in [m-1]$ such that $(i,j), (i,j')\in B$ but $(i,j'')\not\in B$. Consider the link of $a_ia_{i+1}$ in $\TT_B$. It contains the segments $b_jb_{j+1}$ and $b_{j'}b_{j'+1}$ but not the segment $b_{j''}b_{j''+1}$. In particular, it is not connected, while the link of every simplex in a ball must be connected.

For the converse, suppose $B$ is grid-unimodal and let us show that $|\TT_B|$ is a shellable ball.
For this, by induction on the cardinality of $B$, we just need to show that if $B$ is grid-unimodal then there is an $(i_0,j_0)\in B$ such that:
\begin{enumerate}
\item[(a)] $B\setminus \{(i_0,j_0)\}$ is still grid-unimodal (hence $\TT_{B\setminus \{(i_0,j_0)\}}$ is a shellable ball) and
\item[(b)] $|T_{i_0,j_0}|$ intersects $|\TT_{B\setminus \{(i_0,j_0)\}}|$ in a nonempty union of $2$-faces of $T_{i_0,j_0}$.
Observe that this union of $2$-faces is automatically a $2$-ball; indeed, the only nonempty union of facets of a $k$-simplex that is not a $(k-1)$-ball is the whole boundary. But gluing $T_{i_0,j_0}$ along its whole boundary to something that is (by inductive hypothesis) a $3$-ball would produce a $3$-sphere contained in the $3$-ball $\TT$, which is impossible.
\end{enumerate}
Any $(i,j)\in B$ satisfies (b), unless it is the only one in its row and column, which, as $B$ is grid-connected, can happen if and only if $B=\{(i,j)\}$, where the claim is trivial. Assume $|B|>1$ and
let $i_0$ be the maximum first index such that $B\cap (\{i\}\times [m-1])$ is not empty.
Let $j_{\min}$ and $j_{\max}$ be the minimum and maximum second indices in $B\cap (\{i_0\}\times [m-1])$.
Both of $(i_0,j_{\min})$ and $(i_0,j_{\max})$ satisfy property (b) and they almost satisfy property (a):
removing either of them maintains the property that $B\cap (\{i\}\times [m-1])$ and $B\cap ([n-1]\times \{j\})$ are intervals, for every $i$ and $j$.
The only thing that could fail is that $B$ may become grid-nonconnected for one of the choices.
But this cannot happen for both choices: If it happens for one choice, say $(i_0,j_{\min})$, that implies that $(i_0-1,j_{\min})$ is the maximal element of $B\cap \{i'_0-1\}\times [m-1]$, in which case we can safely remove $(i_0,j_{\max})$.
\end{proof}

With this in mind, we present now our first construction:

\begin{construction}
\label{cons:dim3-holes4}
Consider the join $\TT$ of two paths of lengths $n-1$ and $m-1$, with vertices $a_1,\dots,a_n, b_1,\dots, b_m$, as above. For each
$k\in \{1,\dots,\lceil (n+m)/4\rceil$ let
\[
B_k:=\{(i,j): i+j\in[4k-3,4k]\}.
\]
By part (3) of Lemma~\ref{lemma:grid_convexity}, $|\TT_{B_k}|$ is a $3$-ball, for every $k$, since $B_k$ is grid-unimodal. Let $S_k$ be the set of all tetrahedra in $B_k$ with $i+j\in\{0,3\}\pmod 4$,
 which all have more than one triangle in $\partial B$. See Figure~\ref{fig:holes4}. We leave it to the reader to check that $S_k$ is a compatible system of tetrahedra in $B_k$. Hence,
\begin{figure}[htb]
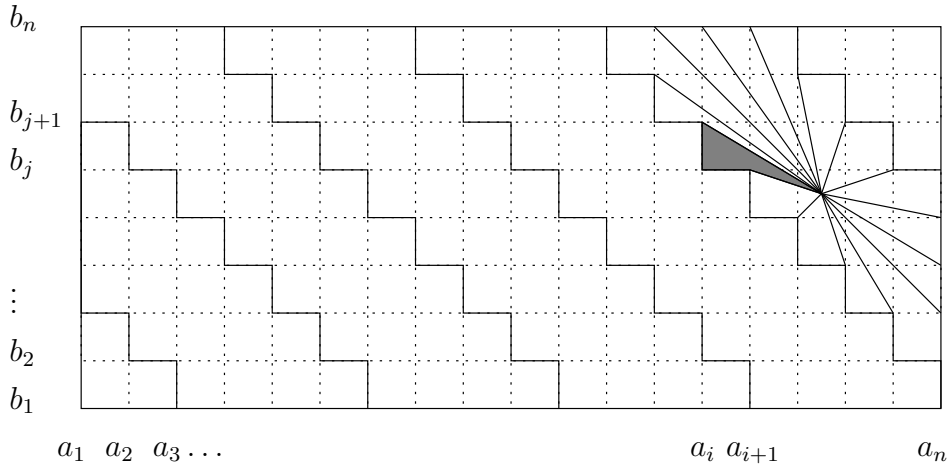
\caption{Construction~\ref{cons:dim3-holes4}}
\label{fig:holes4}
\end{figure}
substituting $B'_1,\dots, B'_{k_{\max}}$ (where $k_{\max}= \lceil (n+m)/4\rceil$) for the old $B_1,\dots, B_{k_{\max}}$ produces a polyhedral complex $B'$ with about $nm/2$ triangular bipyramids (every tetrahedron $T_{ij}$ with $i+j\in \{0,3\} \pmod 4$ produces a bipyramid).
Hence, we get the following statement, where we take $m=n$ for simplicity:

\begin{proposition}
\label{prop:holes4}
$S_i$-filling $B_i$ for each $1\leq i\leq \lceil\frac{n}{2}\rceil$ in Construction \ref{cons:dim3-holes4}
produces
a polyhedral $3$-ball with $5n/2+o(n)$ vertices consisting of  $2n+o(n)$ tetrahedra and $n^2/2+o(n)$ bipyramids.

These bipyramids can be triangulated each in two ways, independently, so that we can get $2^{n^2/2+o(n)}$ different simplicial balls (on labeled vertices).
\end{proposition}

Coning the boundary of the above polyhedral ball $B'$ (or the triangulations of it) to a new vertex, we also get:

\begin{theorem}
\label{thm:holes4}
There is a polyhedral $3$-sphere with $5n/2+o(n)$ vertices consisting of $6n+o(n)$
tetrahedra and $n^2/2+o(n)$ bipyramids.%
\footnote{More can be done: every old tetrahedron in the ball $B'_1\cup \dots \cup B'_{\max}$ was incident to one boundary triangle of $\TT$, so it can be glued together with the tetrahedron at the other side to make more bipyramids,
and the final number of tetrahedra remaining is about $2n$, four outside the corners of each hole $B_k$. Playing a bit with the tetrahedra in corners of each hole one can in fact  get rid off \emph{all} the tetrahedra and get a polyhedral sphere consisting only of bipyramids.}
\end{theorem}

This implies (plugging $N=5n/2+o(n)$ and dividing the number of labeled triangulations by $N!$ to get a lower bound for the number of combinatorial types):

\begin{corollary}
\label{cor:holes4_tri}
For small $\epsilon>0$ and $N$ large enough,
there are at least  $2^{(\frac{2}{25}-\epsilon)N^2}>1.0570^{N^2}$ (PL-) triangulations of the $3$-ball and of the $3$-sphere having $N$ vertices.
\end{corollary}

\end{construction}

\begin{remark}
\label{cons:dim3-holes3}
A slightly larger number of bipyramids (in terms of vertices) can be obtained with a slight modification, except now not all of them can be triangulated independently. (In particular, the number of spheres that this construction produces is slightly worse).
For this, do exactly as before, but define each $B_k$ to be
\[
B_k:=\{(i,j): i+j\in[3k-2,3k]\},
\]
so that now $k$ ranges from $1$ to $k_{\lceil (n+m)/3\rceil}$, and let $S_k$ consist of the tetrahedra in $B_k$ with $i+j\in\{0,2\}\pmod 3$.
This produces about $2n/3$ holes (we assume $n=m$) and about $2n^2/3$ bipyramids. See Figure~\ref{fig:holes3}.
\begin{figure}[htb]
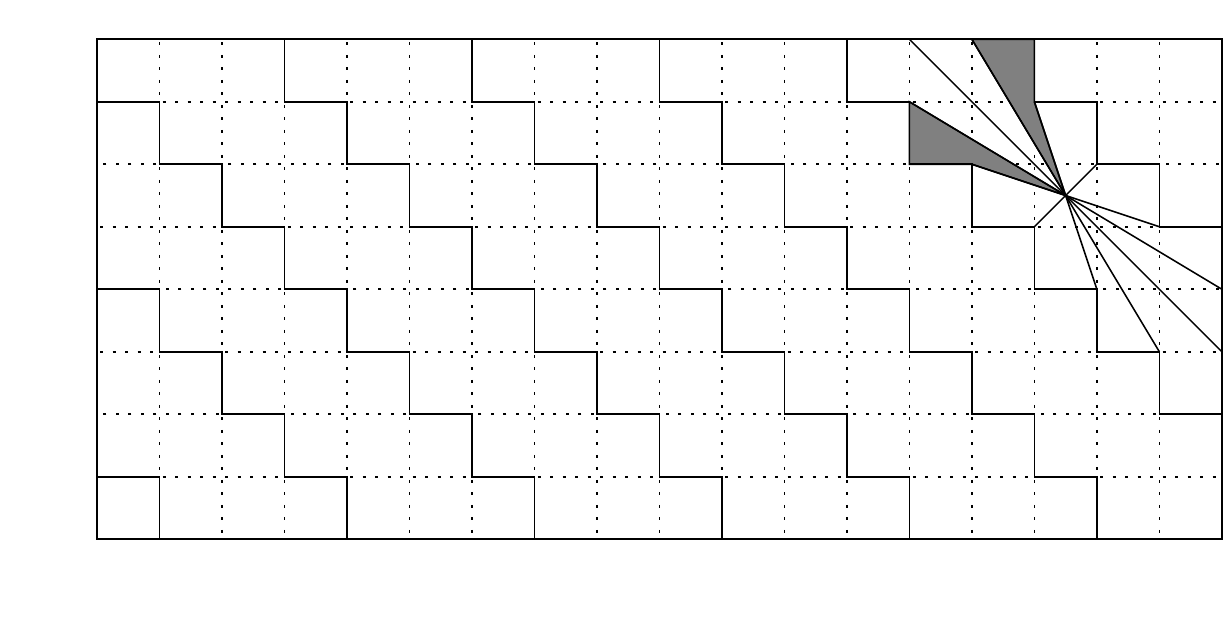
\caption{Construction~\ref{cons:dim3-holes3}}
\label{fig:holes3}
\end{figure}
But, as said above, now the bipyramids cannot be independently triangulated because $S_k$ is not a system of compatible tetrahedra.
Rather, most of the tetrahedra come in pairs with the same missing edge (so that about half of them could be triangulated independently). Still, letting $N=8n/3$
in this construction (of Remark~\ref{cons:dim3-holes3}) we get $3$-spheres with $N$ vertices and about  ${3N^2/32}$ bipyramids, which is greater than ${2N^2/25}$.
\end{remark}

\subsection{Many $4$-polytopes}

We now want to do a geometric construction, rather than merely topological, to get many \emph{polytopal} or many \emph{geodesic} $3$-spheres.
For this, we need our holes to be starconvex, so that their whole boundary is seen from the new point inserted.

\begin{construction}
\label{cons:dim3-subgrid}
Let  $n=kl+1$ for an odd $k$ and arbitrary $l$, and divide the whole grid of $(n-1)\times (n-1)$ rectangles into $l^2$ subgrids of $k\times k$ rectangles each. Inside each subgrid consider a hole $B_{i,j}$ ($i,j\in[l]$) in the shape of an aztec diamond and do the same construction as above, letting $S_{i,j}$ to be the set of all tetrahedra having a triangle in $\partial B_{i,j}$, so that $|S_{i,j}|=2k-2$.
See Figure~\ref{fig:holes-aztec}.
Each $S_{i,j}$ forms a compatible system of tetrahedra.
Now, by Lemma \ref{lemma:grid_convexity}(2),
all the bipyramids obtained are actually convex (which implies they can be triangulated independently) so we get that:

\begin{figure}[htb]
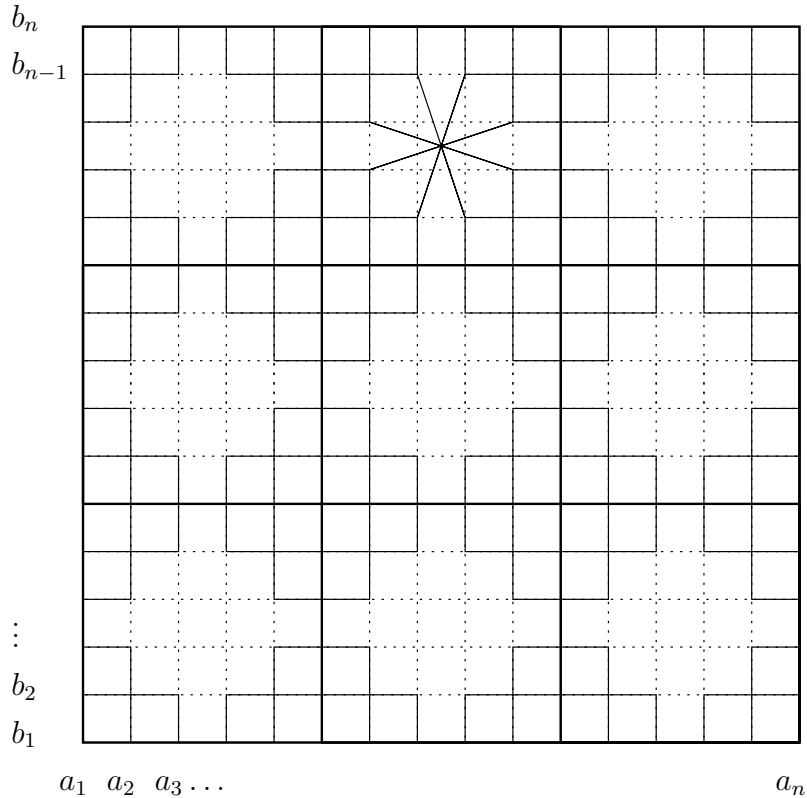
\caption{Construction~\ref{cons:dim3-subgrid}}
\label{fig:holes-aztec}
\end{figure}

\begin{theorem}
\label{thm:holes_aztec}
Let $n=kl+1$.
The point configuration in $\mathbb{R}^3$ consisting of the $2kl+2$ points $a_1,\dots,a_n, b_1,\dots,b_n$ plus the $l^2$ centers of the aztec diamonds has:
\begin{enumerate}
\item A polyhedral subdivision $\T$ containing $(2k-2)l^2$ triangular bipyramids, obtained by $S_{i,j}$-filling $B_{i,j}$ for all $i,j\in [l]$.
\item  Moreover, by an appropriate choice of the points $a_i$ and $b_j$ this subdivision can be made regular.
\end{enumerate}
\end{theorem}

The only part in the statement that does not directly follow from Lemmas~\ref{lemma:one_hole} and~\ref{lemma:grid_convexity} is regularity of $\T$, which depends on how the  $a_i$'s and $b_j$'s are chosen. The  assumption needed on them is rather weak (see the proof below) and, for example, the most natural point set, having $a_i=(i,0,1)$ and $b_j=(0,j,-1)$ satisfies it.

\begin{proof}
Let us  make concrete the assumptions needed on the $a_i$'s and $b_j$'s. We  assume that the $l$ subconfigurations  $a_{ki+1},\dots,a_{ki+k+1}$ ($i=0,\dots, l-1$) corresponding to the $l$ subgrids are translated from one another and centrally symmetric (see the lemma below), with the same assumptions on the $b$'s.

We first show the following lemma about a square grid in the plane:

\begin{lemma}\label{lem:4.7}
Let $k$ be odd and let $X=\{(x_i,y_j) : i=-k,\dots,k, j=-k,\dots,k,\ i,j \text{ odd}\} \cup \{(0,0)\}\subset\R^2$ for some choices of $x_i$'s and $y_j$'s satisfying:
\[
x_{-k} < x_{2-k} < \dots < x_k,\quad
y_{-k} < \dots < y_k,\quad
x_{-i}=-x_{i},\quad
y_{-j}=-y_{j}.
\]

Consider the subdivision of $X$ consisting of the $(k^2-1)/2$ grid rectangles outside the aztec diamond together with the $2k-6$ quadrilaterals and $4$ pentagons obtained by joining $(0,0)$ to the convex pieces of the boundary of the aztec diamond.
Then, this subdivision is regular, with respect to a lifting vector $\omega:X\to \R$ satisfying the following properties:
\begin{itemize}
\item On $\{(x_i,y_j)\}_{i,j}$, $\omega$ splits as a sum $\omega_{i,j}=\alpha_i + \beta_j$.
\item The $\alpha$'s and $\beta$'s are symmetric ( $\alpha_i=\alpha_{-i}$, $\beta_j=\beta_{-j}$).
\end{itemize}
\end{lemma}

\begin{proof}
Since the original configuration and the point set are symmetric under horizontal and vertical reflection, we concentrate on the positive values of $i$ and $j$ and then reflect the solution.

We first give height $0$ to the middle point $(0,0)$ and choose one and the same gradient for all the edges incident to it. Put differently, we draw a small horizontal circle $C$ centered around $(0,0,1)$ and give to each point $(x_i,y_{k-1-i})$ ($i\in\{1,3,\dots,k\}$) the height $\omega_{i,k-1-i}$ necessary to have the edge $(0,0,0)$, $(x_i,y_{k-1-i},\omega_{i,k-1-i})$ meet $C$.

These conditions determine the rest of values of $\omega_{i,j}$ as follows:
\begin{itemize}
\item For each $i=1,3,\dots, k-2$,
coplanarity of the lift of the quadrilateral with vertices $(0,0)$, $(x_i,y_{k+1-i})$, $(x_i,y_{k-1-i})$ and $(x_{i-2},y_{k+1-i})$ gives us a unique choice of lift for the height of each point $(x_i,y_{k+1-i})$.
\item The same happens for the height of $(x_1,y_k)$ and $(x_k,y_1)$ taking into account coplanarity of the lift of the pentagons
\[
\{(0,0), (x_{-1},y_k), (x_1,y_k), (x_{-1},y_{k-2}), (x_1,y_{k-2})\}
\]
 and
 \[
 \{(0,0), (x_k,y_{-1}), (x_k,y_1), (x_{k-2},y_{-1}), (x_{k-2},y_{1})\}.
 \]
(Here, the values $\omega_{-1,k-2}=\omega_{1,k-2}$ and $\omega_{k-2,-1}=\omega_{k-2,1}$ are known by symmetry).

\item With this, for an arbitrary choice of $\alpha_1$ (say $\alpha_1=0$) and taking $\beta_k=\omega_{1,k}-\alpha_1$ we can uniquely compute all the other $\alpha_i$'s and $\beta_{j}$'s ($i=3,\dots,k$; $j=k-2,\dots,1$) as follows:
\[
\alpha_i= \omega_{i,k+1-i} -\omega_{i-2,k+1-i} + \alpha_{i-2},
\]
\[
\beta_{k+1-i}= \omega_{i-2,k+1-i} -\omega_{i-2,k+3-i} + \beta_{k+3-i}.
\]
\end{itemize}

It is clear that this choice of $\omega$'s lifts each quadrilateral of our subdivision (even more so, also the grid quadrilaterals that are not in our subdivision) to be still planar. It is also clear from the construction that the $\alpha_i$'s are a convex function on the $x_i$'s and the $\beta_i$'s on the $y_i$'s (convexity for a consecutive triple $x_{i-1}$, $x_i, x_{i+1}$ follows from the fact that the star of $(0,0)$ is lifted convex, which was our first requirement). This means that, if we remove $(0,0)$ from the configuration, the regular triangulation produced by $\omega$ is precisely the $k\times k$ rectangular grid. Since the heights are chosen so that $(0,0)$ is lifted coplanar to all the quadrilaterals
\[
\{(x_i,y_{k+1-i}), (x_i,y_{k-1-i}), (x_{i-2},y_{k+1-i}), (x_{i-2},y_{k-1-i})\},
\]
the result follows.
\end{proof}

We now finish the proof of Theorem~\ref{thm:holes_aztec}, constructing the regular subdivision in two steps in the spirit of~\cite[Lemma 2.3.16]{deLoeraRambauSantos}, which basically says the following: if $\omega$ and $\omega'$ are two lifting functions on a point configuration $A$ then, for $\epsilon>0$ sufficiently small, the regular subdivision of $A$ produced by $\omega+\epsilon \omega'$ refines the regular subdivision produced by $\omega$ alone in the way indicated by $\omega'$ (that is, each cell of the first regular subdivision is refined as the second lifting function, restricted to the points in that cell, was used to construct a regular subdivision of that cell).

So, denote by $A$ the point configuration in the theorem, containing all the $a_i$'s, $b_j$'s ($i,j\in [kl+1]$) and centers of the subgrids, which we denote $o_{i,j}$ ($i,j\in [l]$).
Consider first a lifting function $\omega$ on $A$ that produces the $l\times l$ ``coarse grid'' (or, more precisely, the corresponding ``coarse join of two paths'') as a regular subdivision. That is to say, consider any convex function on the points of the form $a_{1+ik}$, $i=0,\dots,l$, another (or the same) on the points
$b_{1+ik}$, and interpolate linearly on the rest of the points $a_i$ and $b_j$, as well as in the centers $o_{i,j}$, so that all the $2k+3$ points corresponding to the same subgrid are lifted to lie in a hyperplane. The regular subdivision obtained at this point is the join of two paths of length $l$, where each tetrahedron in this join corresponds to a rectangle in the coarse subgrid. We call them coarse tetrahedra.

We now perturb $\omega$ to an $\omega+\epsilon \omega'$ for a sufficiently small $\epsilon>0$, taking $\omega'$ to be $0$ on the center-points of aztec diamonds and twice the $\alpha$'s and $\beta$'s from Lemma \ref{lem:4.7}, in the $a_i$'s and $b_j$'s of each of the $l^2$ coarse tetrahedra. In this way, the height given to the midpoint $(a_i+b_j)/2$ of an $a_i$ and a $b_j$ (which is a vertex in the corresponding subgrid) is exactly the $\omega_{i,j}$ in the lemma.
\end{proof}

Taking $k=l$, which implies $N\sim 3l^2$ points in $A$ in this construction, gives:

\begin{corollary}
\label{coro:holes_aztec}
For every $\epsilon>0$ and $N$ large enough:
\begin{enumerate}
\item There are $4$-polytopes with $N\ (\sim3l^2)$ vertices and at least $\frac{2-\epsilon}{3^{3/2}}N^{3/2}\ (\sim2kl^2)$ bipyramidal facets.
\item There are at least $2^{(2-\epsilon)(N/3)^{3/2}}$ geodesic triangulations of the $3$-sphere with $N$ vertices.
\end{enumerate}
\end{corollary}
\begin{proof}
Part (1) is obtained by taking the convex hull of the lifted points from Theorem \ref{thm:holes_aztec}. Part (2) is obtained by considering the complete fan defined by the lifted points, call this set $P$, together with another point $x$ such that $-x$ is in the interior of one of the cones defined by $P$.
\end{proof}
\end{construction}

It is clear from our proof of regularity in Theorem~\ref{thm:holes_aztec} that if we slightly increase the height of the central point in each hole (the point $(0,0)$ of the proof) maintaining all of the others constant, then all the bipyramids corresponding to \emph{quadrilaterals} (not to \emph{pentagons}!) are refined in the same way, with three tetrahedra around an edge of degree three.
 In particular, if we do this simultaneously on the central points of all aztec holes, we get a regular triangulation that contains $\Omega(N^{3/2})$ edges of degree $3$:

\begin{corollary}
\label{coro:degree_3}
There are simplicial $4$-polytopes on $N$ vertices with $\Omega(N^{3/2})$ edges of degree three.
\end{corollary}

This disproves a conjecture of Ziegler.
Note that the same can be applied in Construction~\ref{cons:dim3-holes4}, yielding $3$-spheres with $\Theta(N^2)$ edges of degree three.

\begin{remark}[Asymptotics of the number of geodesic spheres]
There is a variation that allows to basically double the exponent obtained for the number of triangulations in part (2) of Corollary~\ref{coro:holes_aztec}.
Observe that in each $k\times k$ subgrid we can, instead of fixing an Aztec diamond, consider an arbitrary star convex region.
The exterior of that region is left subdivided in the grid way and the interior is coned (simplicially) to the center.
Hence, we get $s^{l^2}$ as a lower bound to the number of triangulations of this configuration, where $s$ is the number of subcomplexes of the $k\times k$ grid that are grid-starconvex with respect to the center.
We claim $s$ to be about $2^{4k}$.
For this, observe that in each of the four $((k-1)/2\times (k-1)/2)$-subgrids partitioning of the $k\times k$ grid we can independently choose among the ${k-1 \choose (k-1)/2}\in \Theta(2^k/\sqrt{k})$ monotone paths. This gives (taking $k=l$ and neglecting the $\sqrt{k}$  in the denominator):

\begin{corollary}
\label{coro:holes_aztec_paths}
For any $\epsilon>0$ and $N$ large,
there is a point configuration with $N\ (~3l^2)$ points in $\R^3$ having at least $2^{(4-\epsilon)(N/3)^{3/2}}$ different triangulations.
\end{corollary}

\end{remark}

\begin{remark}[Three-dimensional point sets with $N^{\Omega(N)}$ regular triangulations]
\label{rem:many-triangulations}
We can also adapt the previous construction to show three-dimensional point sets with many \emph{regular} triangulations.
Suppose that we slightly perturb the position of the central points in the holes,
in the (lifted) regular polyhedral subdivision $\mathbf{T}$ of Theorem~\ref{thm:holes_aztec}(2),
to positions such that when we move them upwards till they disappear from the regular subdivision
(i.e., till they are above all hyperplanes defined by the rectangles in the hole)
they cross the planes containing different (lifted) rectangles at different heights. Then we get about $k^2/2$ different regular triangulations\footnote{This can be improved to $k^2$: instead of starting with a hole in the shape of an aztec diamond, start with the whole $k\times k$ subgrid as a hole.} in each subgrid, and we can choose them independently.
Indeed, as the refinement $\mathbf{T}$ of the lifting of the coarse rectangles was obtained by perturbation to the centers of the holes, moving the different centers upwards can be done independently.
Taking $l=k^2$ we have $N=2kl+2+l^2\sim l^2= k^4$ and we get:

\begin{corollary}
\label{coro:many-triangulations}
The $N$-point set so obtained in $\R^3$ has at least $(k^2/2)^{l^2}\in N^{\frac{N}{2} - o(N)}$ regular triangulations.
\end{corollary}

Observe that $N^{\Omega(N)} =2^{\Omega(N\log N)}$ matches the maximum number $ 2^{O(N\log N)}$ of regular triangulations that a point set can have, in any dimension. A construction already giving this number of regular triangulations, but for a point set of dimension four (the vertices of $C(N/2,3)\times [0,1]$)
appears in~\cite[Theorem 7.2.10]{deLoeraRambauSantos}.

From Corollary~\ref{coro:many-triangulations} also follows a lower bound of $N^{\frac{N}{2} - o(N)}$ for the number of (labelled) combinatorially different $4$-polytopes. This bound, however, is worse than the  bound of $N^{2{N} - o(N)}$ obtained by Padrol in~\cite[Cor. 6.9]{Padrol13:manypolytopes}.
\end{remark}

\subsection{Even more $3$-spheres in the cyclic $4$-polytope}

Cyclic polytopes have the maximum possible number of simplices among simplicial polytopes (or simplicial spheres, for that matter) of each given dimension and number of vertices. This suggests that they can lead to asymptotically better constructions than the ones presented so far, and this is indeed the case. To simplify our exposition, let us first relate the boundary complex of the cyclic $4$-polytope to the join of two paths and to the representation of it as a grid that we have used in the previous subsection.

\begin{figure}[h]
\includegraphics[scale=.75]{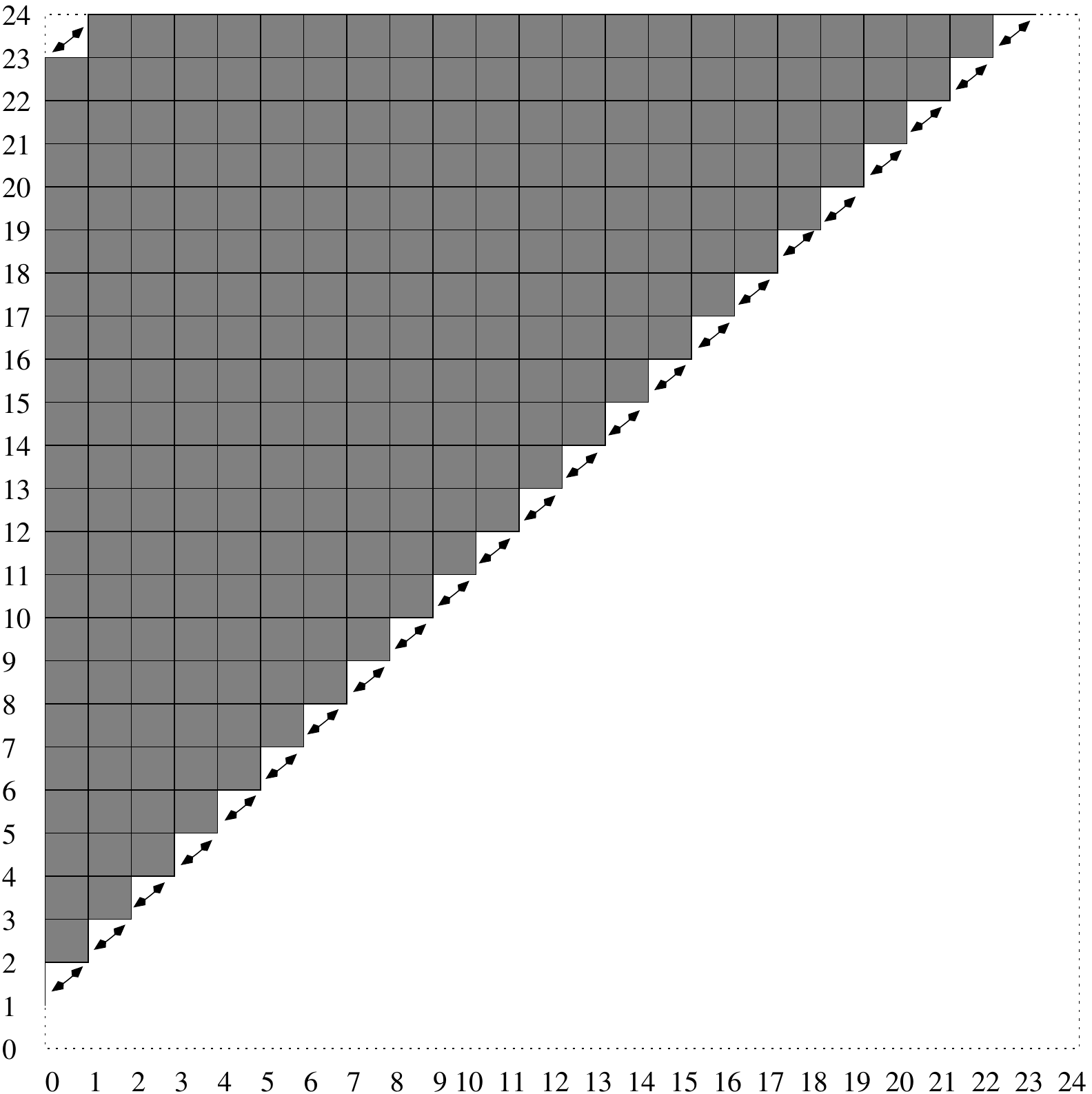}
\caption{The boundary complex of $C(24,4)$}
\label{fig:C_24_4}
\end{figure}

In what follows we let $n\in \N$ and consider the cyclic polytope $C(4n,4)$.
The facets of it are the sets of the form $\{i,i+1,j,j+1\}$ with, without loss of generality, $0\le i<i+1<j\le 4n-1$. Here and in what follows, indices are regarded modulo $4n$. In particular, vertices $0$ and $4n$ are identified. In tetrahedra using $(0,1)$ as one of their pairs we have $i=0$ while in simplices using $(4n-1,4n)$ we have $j=4n-1$.

We can represent these facets as squares in the upper half of a grid, as done in Figure~\ref{fig:C_24_4} for $n=6$, where the tetrahedron
$\{i,i+1,j,j+1\}$  appears as the square $[i,i+1]\times[j,j+1]$. Since $0=24$, the left edge $(0,1)$ and the top edge $(23,24)$ in Figure~\ref{fig:C_24_4} are considered glued to one another. There are the following two interpretations of the figure, one topological and one purely combinatorial:

\begin{itemize}
\item Each vertex $(i,j)$ in the grid can be considered to be the midpoint of the edge $ij$ of the cyclic polytope, so that each shaded square is actually a square (or a parallelogram) cutting the corresponding tetrahedron in two halves. These squares are glued forming a M\"obius strip that is embedded in the boundary of the cyclic polytope.

\item The figure is a representation of the combinatorics of the  complex $K=\partial C(4n,4)$. It is of course not a complete representation, but it is quite faithful in several respects: tetrahedra in $K$ are in bijection to shaded squares, and triangles in $K$ are in bijection to edges in the picture, with a single exception: triangles of the form $(i-1,i,i+1)$ appear in the figure twice, so that the boundary edges of the shaded region in the figure should be considered identified in pairs, as marked by small arrows (but note that these edges are not identified in the M\"obius strip of the previous item; geometrically, they are two different segments in the same triangle and with a single common end-point).
    Finally, all of the $\binom{4n}{2}$ edges of $K$ appear in the figure, as vertices.
\end{itemize}

With this in mind, for each $k\in [n]$ we define the following subcomplex of $\partial C(4n,4)$ (see Figure~\ref{fig:C_24_4_efficient}):
\begin{eqnarray*}
B_k&:=&\{ \{i,i+1,j,j+1\}\in \partial C(4n,4) :\, i+j\in \{4k-2,4k-1,4k,4k+1\} \}\\
&&\setminus \{\  \{4k-2,4k-1,4k,4k+1\}, \ \ 2n+\{4k-2, 4k-1, 4k, 4k+1\}\ \}.\\
\end{eqnarray*}

\begin{lemma}
\label{lemma:cyclic_efficient}
$B_1,\dots, B_n$ are shellable balls with no common tetrahedron.
\end{lemma}

\begin{figure}[h]
\includegraphics[scale=.75]{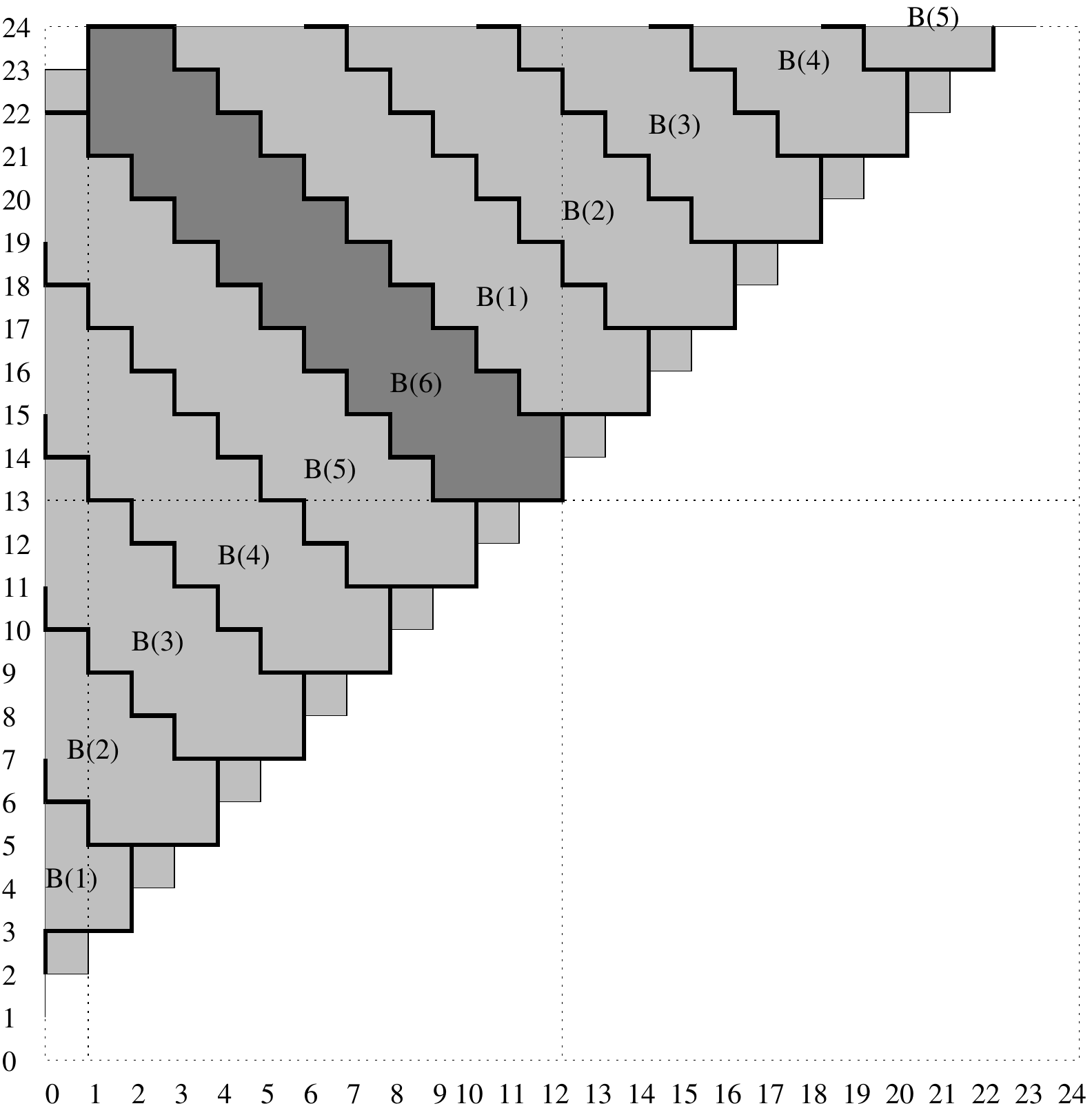}
\caption{More efficient balls, in the cyclic polytope}
\label{fig:C_24_4_efficient}
\end{figure}

\begin{proof}
That no tetrahedron lies in two different $B_i$'s is clear by construction. To show that $B_i$ is a shellable ball we concentrate in $B_n$, which is enough by cyclic symmetry. $B_n$ consists of the following tetrahedra:
\[
B_n:=\{ \{i,i+1,j,j+1\} :\, i\in [1,2n], j \in [2n+1,4n], i+j\in \{4n-2,4n-1,4n,4n+1\} \}\\
\]

In particular (as emphasized for $B_6$ in Figure~\ref{fig:C_24_4}), $B_n$ is contained in the join of paths $[1,2n]*[2n+1,4n]$. As such, we can apply Lemma~\ref{lemma:grid_convexity} to it. Since, as a subset of the grid, it is grid-unimodal, we have that, as a simplicial complex, it is a shellable ball.
\end{proof}

Now, in each $B_k$ we consider the compatible system of tetrahedra $S_k$ consisting of those with $i+j\in \{4k-2,4k+1\}$. Now, $S_k$-filling the hole $B_k$ for all $k$ gives:

\begin{theorem}
\label{thm:holes4_cyclic}
For each large enough $n\in \N$ there is a polyhedral $3$-sphere with
$5n+1$
vertices
consisting of  $\Theta(n)$ tetrahedra and at least $4n^2- o(n^2)$ bipyramids. %
These bipyramids can be triangulated independently, providing at least $2^{4n^2 - o(n^2)}$ different simplicial ({\rm PL}) $3$-spheres.
\end{theorem}

\begin{corollary}
\label{cor:holes4_cyclic}
For $\epsilon>0$ small enough and $N$ large enough,
there are at least $2^{(\frac{4}{25}-\epsilon)N^2}>1.117^{N^2}$ ({\rm PL}-) triangulations of the $3$-sphere with $N$ vertices.
\end{corollary}

\section{Many odd-dimensional spheres}
\label{sec:higherdim}

\subsection{Holes in a multidimensional grid}

The same ideas work without much change in higher odd dimensions.
We start with the join $\TT$  of $d$ paths in $\mathbb{R}^{2d-1}$, of lengths $n_i-1$, for $1\leq i\leq d$.
(Note that throughout this section $d$ is not the dimension of a sphere, but that of a grid. The corresponding balls and spheres have dimension $2d-1$).

For concreteness, define $\TT$ inductively as follows: after embedding the first $d-1$ paths in $\mathbb{R}^{2d-3}$, embed    $\mathbb{R}^{2d-3}\hookrightarrow \mathbb{R}^{2d-1}$ by $x\mapsto (x,0,-1)$, and embed the $d$th path in the line $(0,t,1)$ in $\mathbb{R}^{2d-1}$. With this convention,
$\TT$ is the unique triangulation of the above point configuration consisting of the vertices of the $d$ paths.
The intersection of $\TT$ with the middle fiber of the projection sending each path to a point gives an $(n_1-1)\times \dots \times (n_d-1)$ grid of $d$-cubes $\RR$ (or, rather, orthogonal parallelepipeds, which we call cubes for simplicity).
Let us index the cubes by the elements  $I=(i_1,\dots,i_d)\in[n_1-1]\times\dots\times [n_d-1]$ (their ``bottom corner") and let us index the vertices of the paths as $a_{i}^{(j)}$, with $j=1,\dots, d$ and $i=1,\dots, n_j$.
The cube of index $I$ corresponds to a simplex $T_I\in \TT$ with vertex set
\[
T_I:= \{a_{i_1}^{(1)}, a_{i_1+1}^{(1)},\dots, a_{i_{d}}^{(d)}, a_{i_{d}+1}^{(d)} \}.
\]

We adapt the definitions of grid-convexity, etc. to this case as follows:

\begin{definition} Let $\BB\subset[n_1-1]\times\dots \times [n_d-1]$ represent a union of $(2d-1)$-simplices in $\TT$ or, equivalently, of $d$-cubes in $\RR$.
\label{defi:grid_convexity_d}
\begin{enumerate}
\item We say that $\BB$ is \emph{grid-connected} if we can go from any ${I}\in \BB$ to any ${I'}\in \BB$ moving always from a cube to an adjacent one.
That is, if the dual graph of $\BB$ ($\BB$ as a sub complex of $\RR$) is connected.
\item We say that $\BB$ is \emph{grid-starconvex} from a certain ${I}\in \BB$ if, for every ${I'}\in \BB$ and every $I''\in [I,I']$ we have ${I''}\in \BB$. Here we say that $I''\in [I,I']$ for vectors $I=(i_1,\dots, i_d)$, $I'=(i'_1,\dots, i'_d)$ and $I''=(i''_1,\dots, i''_d)$ if
$i''_k\in [i_k,i'_k]$ for every $k\in [d]$.
\item We say that $\BB$ is \emph{grid-unimodal} if it is grid-connected and $\BB\cap ([n_1-1]\times\dots\times [n_{d_1-1}-1]\times\{i\}\times[n_{d_1+1}-1]\times\dots\times [n_d-1])$ is grid-unimodal, for every $1\leq d_1\leq d$ and every $i\in[n_{d_1}-1]$.
\end{enumerate}
\end{definition}

\begin{lemma}
\label{lemma:grid_convexity_d}
Let $\BB\subset[n_1-1]\times\dots \times [n_d-1]$ and let $\TT_\BB$ denote the union of $(2d-1)$-simplices from $\TT$ represented by $\BB$. Then:
\begin{enumerate}
\item $\BB$ is {grid-connected} if and only if the interior of $\TT_\BB$ is connected.
\item $\BB$ is {grid-starconvex} from ${I}$ if and only if $\TT_\BB$ is star convex from any point in the interior of the simplex $T_{I}$.
\end{enumerate}
\end{lemma}

\begin{proof}
The proofs are the same as those in Lemma~\ref{lemma:grid_convexity}.
\end{proof}

We believe the equivalent of part (3) of Lemma~\ref{lemma:grid_convexity} also to be true: ``If $\BB$ is grid-unimodal then $\TT_\BB$ is  a shellable $(2d-1)$-ball''.
But let us prove this only in the case that we need, the analogue of Construction~\ref{cons:dim3-holes4}.
For this, let $m_1$ and $m_2$ be any two positive integers with
$d\le m_1\le m_2\le -d+\sum_{i=1}^d n_i$
and let $\BB= \BB(d,m_1,m_2) \subset [n_1-1]\times \dots \times [n_d-1]$ be the set of all the cubes with indices $I=(i_1,\dots,i_d)$ having  $i_1+\dots+i_d$ between these two values. That is:
\[
\BB=\BB(d,m_1,m_2):=\{(i_1,\dots,i_d) : m_1 \le i_1+\dots + i_d  \le m_2\}.
\]

\begin{lemma}
\label{lemma:diagonal_dim_d}
Let $\BB=\BB(d,m_1,m_2)$ be as above.
If either $m_2-m_1\ge d$, or $m_1=d$, or $m_2=\sum n_i-d$ then $\TT_\BB$ is a shellable $(2d-1)$-ball.
\end{lemma}

\begin{proof}
Apart of the $(2d-1)$-complex $\TT$ corresponding to the $d$-dimensional grid $[n_1-1]\times\dots \times [n_d-1]$ let us consider the
$(2d-3)$-complex $\TT^0$ corresponding to the $(d-1)$-dimensional grid $[n_1-1]\times\dots \times [n_{d-1}-1]$.
The proof is by induction on $d$ (for $d=1$ the result is trivial) and uses the fact that if a $\BB$ contains only tuples $(i_1,\dots,i_d)$ with a single value of $i_d$ then $\TT_\BB$ is the join of a segment and the complex $\TT^0_{\BB_0}$, where $\BB_0$ is obtained forgetting the last coordinate in each tuple in $\BB$.
In this case, $\BB_0=\BB(d-1,m'_1,m'_2)$ where
$m'_1=m'_1(i_d)=\max(d-1,m_1-i_d)$ and
$m'_2=m'_2(i_d)=\min(-(d-1)+\sum_{j=1}^{d-1}n_j , m_2-i_d)$.

So, let $n_1,\dots, n_d$, and let $m_1,m_2$ as above be given.
Once $d$ is fixed, we use induction on the sum $n_1+\cdots + n_d$. This induction starts with $\sum n_j=2d$,
in which case there is nothing to prove ($\TT$ is a single simplex). We can also assume that:
\begin{itemize}
\item $n_d\ge 3$, or otherwise $\TT_\BB$ is  the join of a segment and a lower-dimensional complex satisfying the hypotheses.
\item $\BB$ contains some $(i_1,\dots,i_d)$ with $i_d=n_d-1$, or otherwise $\BB \subset {[n_1-1]}\times\dots \times [n_{d-1}-1]\times [n_d-2]$, and the result follows by induction on $n_1+\cdots + n_d$.
\end{itemize}

In these conditions, let us define:
\begin{eqnarray*}
\BB_1&=&\{(i_1,\dots,i_d) : m_1 \le i_1+\dots + i_d  \le m_2, i_d\le n_d-2\},\\
\BB_2&=&\{(i_1,\dots,i_d) : m_1 \le i_1+\dots + i_d  \le m_2, i_d= n_d-1\}.
\end{eqnarray*}
Then $\BB_2$ is the join of the segment $[n_d-1,n_d]$ with $\BB(d-1,m'_1(n_d-1),m'_2(n_d-1))$.
By inductive hypothesis, both $\TT_{\BB_1}$ and $\TT_{\BB_2}$ are shellable $(2d-1)$-balls. Their intersection is the join of a vertex and
$\TT^{0}_{\BB_3}$ where
\[
\BB_3=\{(i_1,\dots,i_{d-1}) : m_1-n_d+2 \le i_1+\dots + i_{d-1}  \le m_2-n_d+1\}.
\]
Note that $\BB_3=\BB(d-1,m'_1(n_d-2),m'_2(n_d-1))$ (and if $m_2-m_1\ge d$ then $m'_2-m'_1\ge d-1$), thus, also by inductive hypothesis,
$\TT^{0}_{\BB_3}$ is a shellable $(2d-3)$-ball.
This alone shows that $\TT_\BB=\TT_{\BB_1}\cup \TT_{\BB_2}$ is a $(2d-1)$-ball. To show shellability we distinguish two cases:

\begin{itemize}
\item If $m_1 \le n_d - 2$, then for every $(i_1,\dots,i_{d-1}, n_d-1)\in \BB_2$ we have that $(i_1,\dots,i_{d-1}, n_d-2)\in \BB_1$ and
$(i_1,\dots,i_{d-1})\in \BB_3$. That is, $\BB_2$ is the join of $\BB_3$ with a segment. Then, a shelling order for $\BB$ can be obtained by concatenating any shelling order for $\BB_1$ followed by any shelling order for $\BB_2$.

\item If $m_1 \ge n_d - 1$, then we further subdivide $\BB_2$ into:
\begin{eqnarray*}
\BB'_2&=&\{(i_1,\dots,n_d-1) :  i_1+\dots + n_d-1  = m_1\},\\
\BB''_2&=&\{(i_1,\dots,n_d-1) : m_1 + 1 \le  i_1+\dots + n_d-1  \le m_2\}.
\end{eqnarray*}

See Figure~\ref{fig:HD-shellable} for an illustration.
Now $\BB''_2$ is the join of $\BB_3$ with a segment and, in particular, a shelling order for $\BB_1 \cup \BB''_2$ can be obtained by concatenating any shelling order for $\BB_1$ followed by any shelling order for $\BB''_2$.
\begin{figure}
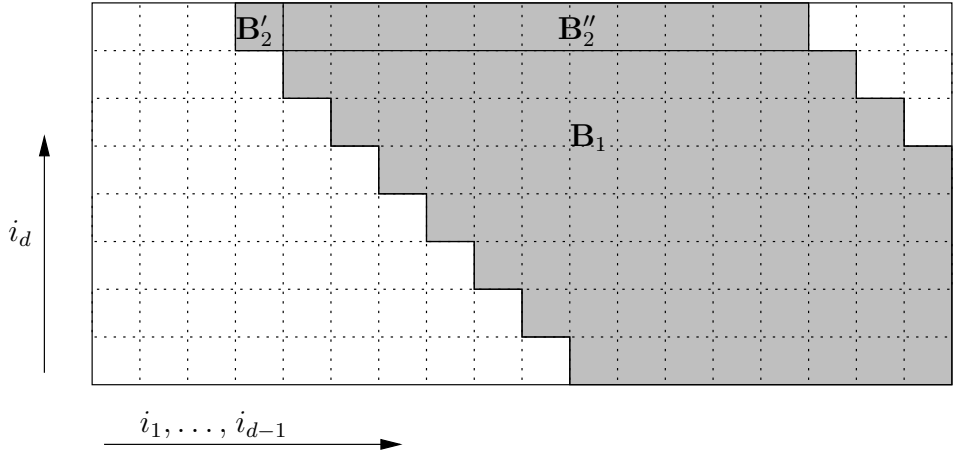
\caption{Last step in the proof of Lemma~\ref{lemma:diagonal_dim_d}}
\label{fig:HD-shellable}
\end{figure}

 After this is done, any ordering of the simplices of $\BB'_2$ completes a shelling order for $\BB$. The reason for this is:
 \begin{itemize}
 \item Any two tuples in $\BB'_2$ differ in at least two indices, since their sum of indices coincide. Thus, the corresponding $(2d-1)$-simplices intersect only in faces of codimension at least two (which, moreover, are already in the complex $\BB''_2$).
 \item Any full-dimensional simplex in $\BB'_2$ intersects $\BB_1 \cup \BB''_2$ in a union of facets.
 \end{itemize}
\end{itemize}
\end{proof}

\subsection{Many  spheres of dimension $2d-1$}

We now look at the analogue of Construction~\ref{cons:dim3-holes4} in higher dimension:

\begin{construction}
\label{cons:dimd-holes}
Let $n_1=\dots=n_d=n$, and in the $[n-1]\times\dots\times [n-1]$ grid $\RR$ consider the sets $\BB_1,\dots \BB_2,\dots,\BB_{\lceil d(n-1)\rceil / (d+2)}$ of the form
\[
\BB_k:=\{(i_1,\dots,i_d) : (k-1)(d+2) \le i_1+\dots + i_d  < k(d+2)\}.
\]

Each such set is, by Lemma~\ref{lemma:diagonal_dim_d}, a ball in the complex $\TT$.
We consider those balls as holes in the complex, that is, we introduce a new vertex $o_k$ in the interior of such ball, cone it to the boundary (so that the ball $\TT|_{\BB_k}$ is changed to another ball $\TT_k$ with the same boundary), and consider the complex obtained as the union of these balls $\TT_k$.
As we did in Construction~\ref{cons:dim3-holes4}, the coning of each point $o_k$ is not done to the individual simplices in the boundary of  $\TT|_{\BB_k}$, but rather to the unions of facets of the same full-dimensional simplex of $\TT|_{\BB_k}$.
In this way, the maximal cells of each $\TT_k$ are not necessarily simplices: some are $2d-1$-simplices, with $2d$ vertices, but others have $2d+1$ vertices.
For each $k$, the $2d-1$-simplices in $\BB_k$ with at least two facets in $\partial\BB_k$ form a compatible family
(their missing faces are the $\tau$'s described below),
thus each $\TT_k$ is a polyhedral ball.
Moreover,
cells belonging to different holes intersect in a single face of both, because the part of each cell in the boundary of a hole is also part of the boundary of a full-dimensional simplex of the original triangulation $\TT$.
That is, the union $\widetilde{\TT}:= \TT_1\cup \TT_2\cup \dots $ is a polyhedral cell complex.

In particular, consider the (about) $n^d\frac{2}{d+2}$ vertices of the grid that happen to lie in the interior of some hole. These correspond to the $(d-1)$-simplices of $\TT$ of the form
\[\tau:=
\{ a_{j_1}^{(1)},\dots, a_{j_{d}}^{(d)}\},
\]
for some points with $\sum j_i \in \{d,d+1\} \pmod{d+2}$. (We are neglecting here some ``boundary effects''. Some of the simplices/points of this form lie in the boundary of $\TT$, so they do not represent grid vertices in the interior of a hole; however, there are at most $o(n^d)$ of them.)

\begin{lemma}
Each of these  $(d-1)$-simplices $\tau$ is the unique missing $(d-1)$-face in one of the cells of $\widetilde{\TT}$ with $2d+1$ vertices. In particular, those cells are combinatorially equivalent to
the free sum of a $(d-1)$-simplex and a $d$-simplex (which is the same as the
cyclic $2d-1$-polytope with $2d+1$ vertices) and thus have two triangulations, one into $d$ full-dimensional simplices (by inserting the missing $d$-face) and one into $d+1$ (by inserting the missing $(d-1)$-face).
\end{lemma}

\begin{proof}
Consider first one simplex $\tau=\{ a_{j_1}^{(1)},\dots, a_{j_{d}}^{(d)}\}$ with $\sum j_i = d \pmod{d+2}$. It is a face of the full-dimensional simplex
\[
\sigma:=\{a_{j_1-1}^{(1)}, a_{j_1}^{(1)},\dots, a_{j_{d}-1}^{(d)}, a_{j_{d}}^{(d)} \},
\]
which corresponds, in $\RR$, to the $d$-tuple $({j_1-1},\dots, {j_{d}-1} )$.
Since this $d$-tuple has $\sum (j_i-1)=\sum j_i - d = 0 \pmod{d+2}$ this simplex  incident to the ``lower diagonal'' boundary of its hole. In particular, all facets of $\sigma$ not containing $\tau$ are in the boundary of the hole, and all facets containing $\tau$ are in the interior. The two triangulations in question consist, respectively of: the facets of $\sigma$ not containing $\tau$, all joined to $o_k$; and $\sigma$ together with the facets of $\sigma$ containing $\tau$, the latter joined to $o_k$.

For a simplex $\tau=\{ a_{j_1}^{(1)},\dots, a_{j_{d}}^{(d)}\}$ with $\sum j_i = d+1 \pmod{d+2}$ the proof is similar,
except now
\[
\sigma:=\{a_{j_1}^{(1)}, a_{j_1+1}^{(1)},\dots, a_{j_{d}}^{(d)}, a_{j_{d}+1}^{(d)} \},
\]
which corresponds, in $\RR$, to the $d$-tuple $({j_1},\dots, {j_{d}} )$ with sum of indices equal to $d+1 \pmod{d+2}$.
$\sigma$ is now incident to the ``upper diagonal'' boundary of its hole, and the same description of the two triangulations of the cell holds.
\end{proof}

With this:

\begin{theorem}
\label{thm:high_d}
Let $d\ge 3$ be fixed and $n$ be large and consider
the polyhedral cell complex $\widetilde{\TT}$ of Construction~\ref{cons:dimd-holes} with
 $dn + \lceil d(n-1)\rceil / (d+2)$ vertices.
Then $\widetilde{\TT}$ is
 a $(2d-1)$-ball  having at least $\frac{2n^d}{d+2}-o(n^d)$ non-simplicial cells. Each such cell can be independently triangulated (without new vertices) in two ways (differing by a bistellar flip).
\end{theorem}

Plugging $N=dn + \lceil d(n-1)\rceil / (d+2)
\ (\sim \frac{d(d+3)}{d+2}n)$ we get:

\begin{corollary}
\label{coro:high_d}
For $\epsilon>0$ and $N$ large,
the number of {\rm PL} $(2d-1)$-balls (or $(2d-1)$-spheres) on $N$ vertices grows at least as
\[
2^{\left(2 \frac{(d+2)^{d-1}}{d^d(d+3)^d} -\epsilon \right)N^d} >
2^{\left( \frac{2}{3d^{d+1}} \right)N^d}.
\]
\end{corollary}

In the last inequality we use that, for every $d\ge 3$,
\[
\frac{d(d+2)^{d-1}}{(d+3)^d} >\frac{1}{3}.
\]
\end{construction}

\subsection{Many geodesic spheres of dimension $2d-1$
}

Similarly, we can generalize construction~\ref{cons:dim3-subgrid}.

\begin{construction}
\label{cons:dimd-subgrid}
Consider $n=lk+1$ with $k$ odd, and divide the $(n-1)^d$ grid into $l^d$ subgrids of $k^d$ $d$-cubes each.

In each subgrid we call \emph{Aztec crosspolytope} the union of cubes at distance less than $k/2$ from the central cube (where distance is measured by adjacency among cubes).
Observe that the total number of cubes in the aztec cross-polytope of dimension $d$ and diameter $k$ is
\[
E_d\left(\frac{k-1}{2}\right):=\sum_{i=0}^d {d\choose i}{\frac{k-1}{2}+i\choose d} = \frac{k^d}{d!}  + O(k^{d-1}),
\]
for $d$ fixed and $k$ large.
($E_d(x)$ is the Ehrhart function of the usual $d$-dimensional cross-polytope, which is computed for example in~\cite[Sect. 2.5]{BeckRobins}).

Aztec crosspolytopes are  grid-starconvex, so we can apply the $\S$-filling construction to them, taking $\S$ to be the set of $2d-1$-simplices corresponding to boundary cubes, which all have more then one facet in the boundary of $\TT_\BB$.
The number of them is
\[
E_d\left(\frac{k-1}{2}\right)- E_d\left(\frac{k-3}{2}\right)= \frac{2k^{d-1}}{(d-1)!} + O(k^{d-2}).
\]

Doing this we obtain a subdivision of a point configuration in $\mathbb{R}^{2d-1}$ with $N= d(lk+1) + l^d$ vertices containing about $2 k^{d-1} l^d / (d-1)!$ facets that are not simplices, but free sums of two simplices.

Taking $k=l^{d-1}$, so that $dn+l^d\sim (d+1)l^d$ we have:

\begin{theorem}
\label{t:complex_k}
\label{thm:aztec-highd}
There is a  subdivision of $N\sim (d+1) l^d$ points in $\mathbb{R}^{2d-1}$
having
\[
\frac{2 k^{d-1} l^d }{ (d-1)!}
\sim \frac{2 l^{d^2-d+1} }{ (d-1)!}
\sim \frac{2}{(d-1)!} \left(\frac{N}{d+1}\right)^{\frac{d^2-d+1}{d}}
> \frac{2N^{\frac{d^2-d+1}{d}}}{(d-1)!(d+1)^{d}}
\]
facets that are not simplices.
\end{theorem}

Observe that for big (but fixed) $d$ this is not much different from the maximum number of facets of a $2d-1$-sphere or ball with $N$ vertices, which is $\frac{N}{N-d}\binom{N-d}{d}\sim \frac{1}{d!} N^d$.
Moreover, triangulating these cells independently we have:

\begin{corollary}
\label{t:sphere_k}
\label{coro:aztec-highd}
There are $2^{N^{\Omega\left(N^{d-1+\frac{1}{d}}\right)}}$ combinatorially different geodesic triangulations of the $(2d-1)$-sphere.
\end{corollary}

\begin{remark}
We will not get into  details, but the subdivision of Theroem~\ref{thm:aztec-highd} can be made regular with the same ideas used in Theorem~\ref{thm:holes_aztec} (which is the case $d=2$ of this same construction). This implies the existence of $2d$-polytopes with $N$ vertices and with more than
\[
\frac{2N^{d-1+\frac{1}{d}}}{(d-1)!(d+1)^{d}} \in \Omega(N^{d-1+\frac{1}{d}})
\]
facets that are not simplices.

One could also use the ideas of Remark~\ref{rem:many-triangulations} and Corollary~\ref{coro:many-triangulations} in this higher dimensional context, obtaining from this construction the existence of about $\left(\frac{N}{d+1}\right)^{\frac{d-1}{d+1}N} $ combinatorially different $2d$-polytopes. But, as already happened for $d=2$, this number is much smaller than the one obtained by Padrol~\cite{Padrol13:manypolytopes}.
\end{remark}

\end{construction}

\textbf{Acknowledgments}:
We thank Gil Kalai and G\"{u}nter Ziegler for helpful comments on earlier versions of this paper.
%\newpage

\bibliography{gbiblio}{}

\begin{thebibliography}{10}

\bibitem{Alon:fewPolytopes}
Noga Alon.
\newblock The number of polytopes, configurations and real matroids.
\newblock {\em Mathematika}, 33(1):62--71, 1986.

\bibitem{BeckRobins}
Matthias Beck and Sinai Robins.
\newblock {\em Computing the Continuous Discretely}.
\newblock Undergraduate Texts in Mathematics. Springer-Verlag, Berlin, 2007.

\bibitem{deLoeraRambauSantos}
Jes\'us~A. De~Loera, J\"org Rambau, and Francisco Santos.
\newblock {\em Triangulations: Structures for Algorithms and Applications}.
\newblock Springer, 2010.

\bibitem{Dey93:triangulations}
Tamal~Krishna Dey.
\newblock On counting triangulations in {$d$} dimensions.
\newblock {\em Comput. Geom.}, 3(6):315--325, 1993.

\bibitem{GoodmanPollack:fewPolytopes-86}
Jacob~E. Goodman and Richard Pollack.
\newblock Upper bounds for configurations and polytopes in {${\mathbf{R}}\sp
  d$}.
\newblock {\em Discrete Comput. Geom.}, 1(3):219--227, 1986.

\bibitem{Grunbaum}
Branko Gr{\"u}nbaum.
\newblock {\em Convex polytopes}, volume 221 of {\em Graduate Texts in
  Mathematics}.
\newblock Springer-Verlag, New York, second edition, 2003.
\newblock Prepared and with a preface by Volker Kaibel, Victor Klee and
  G\"unter M.\ Ziegler.

\bibitem{Joswig-Ziegler}
M.~Joswig and G.~M. Ziegler.
\newblock Neighborly cubical polytopes.
\newblock {\em Discrete Comput. Geom.}, 24(2-3):325--344, 2000.
\newblock The Branko Gr{\"u}nbaum birthday issue.

\bibitem{Kalai-manyspheres}
Gil Kalai.
\newblock Many triangulated spheres.
\newblock {\em Discrete Comput. Geom.}, 3(1):1--14, 1988.

\bibitem{Kalai-Kyoto}
Gil Kalai.
\newblock Open problems on convex polytopes {I}'d love to see solved.
\newblock Presentation at the ``Workshop on convex polytopes'', Kyoto, July
  2012.
\newblock \url{http://www.gilkalai.files.wordpress.com/2012/08/kyoto-3.pdf}.

\bibitem{Moise}
Edwin~E. Moise.
\newblock Affine structures in 3-manifolds. v. the triangulation theorem and
  hauptvermutung.
\newblock {\em Annals Math.}, 56:96–--114, 1952.

\bibitem{Munkres}
James~R. Munkres.
\newblock {\em Elements of algebraic topology}.
\newblock Addison-Wesley Publishing Company, Menlo Park, CA, 1984.

\bibitem{Padrol13:manypolytopes}
Arnau Padrol.
\newblock Many neighborly polytopes and oriented matroids.
\newblock {\em Discrete Comput. Geom.}, 50(4):865--902, 2013.

\bibitem{Pfeifle-Ziegler:many3-spheres}
Julian Pfeifle and G{\"u}nter~M. Ziegler.
\newblock Many triangulated 3-spheres.
\newblock {\em Math. Ann.}, 330(4):829--837, 2004.

\bibitem{Stanley:CohenMacaulayUBC-75}
Richard~P. Stanley.
\newblock The upper bound conjecture and {C}ohen-{M}acaulay rings.
\newblock {\em Studies in Appl. Math.}, 54(2):135--142, 1975.

\bibitem{Ziegler}
G{\"u}nter~M. Ziegler.
\newblock {\em Lectures on polytopes}, volume 152 of {\em Graduate Texts in
  Mathematics}.
\newblock Springer-Verlag, New York, 1995.

\end{thebibliography}
\bibliographystyle{plain}
\end{document}